\documentclass{article}
\usepackage[utf8]{inputenc}
\usepackage{amsmath,amssymb,amsthm,graphicx,mathrsfs,url}
\usepackage{graphicx} 
\usepackage{float} 
\usepackage{subfigure} 
\usepackage[usenames,dvipsnames]{color}
\usepackage[colorlinks=true,linkcolor=blue,citecolor=Green,urlcolor=Blue]{hyperref}
\usepackage[super]{nth}
\usepackage[open, openlevel=2, depth=3, atend]{bookmark}
\hypersetup{pdfstartview=XYZ}

\def\?[#1]{\textbf{[#1]}\marginpar{\Large{\textbf{??}}}}

\numberwithin{equation}{section}

\newtheorem{theorem}{Theorem}[section]
\newtheorem{lemma}[theorem]{Lemma}
\newtheorem{proposition}[theorem]{Proposition}
\newtheorem{corollary}[theorem]{Corollary}

\newtheorem{remark}[theorem]{Remark}

\newcommand{\mc}{\mathcal}

\newcommand{\cc}{\mathbb{C}}

\newcommand{\la}{\lambda}
\newcommand{\eps}{\epsilon}

\newcommand{\til}{\widetilde}

\newcommand{\demi}{\tfrac{1}{2}}

\let\Im=\Imag

\let\Re=\Real

\DeclareMathOperator{\Vol}{Vol}

\renewcommand{\tilde}{\widetilde}          
\DeclareMathSymbol{\leqslant}{\mathalpha}{AMSa}{"36} 
\DeclareMathSymbol{\geqslant}{\mathalpha}{AMSa}{"3E} 
\DeclareMathSymbol{\eset}{\mathalpha}{AMSb}{"3F}     
\renewcommand{\leq}{\;\leqslant\;}                   
\renewcommand{\geq}{\;\geqslant\;}                   
\newcommand{\dd}{\text{\rm d}}             

\newcommand{\D}{\mathbb{D}}
\newcommand{\R}{\mathbb{R}}

\renewcommand{\H}{\mathbb{H}}
\newcommand{\N}{\mathbb{N}}

\newcommand{\E}{\mathbb{E}}

\def\eps{\varepsilon}

\def\bi{\begin{itemize}}
\def\ei{\end{itemize}}

\def\bnum{\begin{enumerate}}
\def\enum{\end{enumerate}}
\def\<#1{\langle #1 \rangle}

\title{Liouville conformal field theory on Riemann surface with boundaries
}
\author{Baojun Wu}

\begin{document}
\maketitle
\begin{abstract}
In this note, we give a unified rigorous construction for the Liouville conformal field theory on compact Riemann surface with boundaries for $\gamma\in (0,2]$ and prove a certain type of Markov property. We also give some bulk-boundary fusion-type estimates in the boundary LCFT. This note will serve as a reference for a program leading to the conformal bootstrap for the Riemann surface with boundaries and several related projects.
\end{abstract}
\tableofcontents
\footnotesize



\normalsize




\section{Introduction}
Given a two dimensional connected compact Riemannian surface $(\Sigma,g)$ with boundary $\partial \Sigma$, we define the Liouville functional on $C^1$ maps $\varphi:\Sigma\to\R$ by
\begin{equation}\label{Laction}
S_L(g,\varphi):= \frac{1}{4\pi}\int_{\Sigma}\big(||d\varphi||_g^2+QR_g \varphi  + 4\pi \mu e^{\gamma \varphi  }\big)\,{\rm dv}_g+\frac{1}{2\pi}\int_{\partial{\Sigma}}\big(Qk_g \varphi +2\pi \mu_{\partial} e^{\frac{\gamma}{2}\varphi}\big)\,{\rm d\lambda}_g
\end{equation}
where $Q>0,\gamma>0$, bulk cosmology constant $\mu\geq 0$ and boundary cosmology constant $\mu_\partial\geq 0$, we suppose $\mu+\mu_\partial>0$ otherwise the theory is not renormalizable. Here $R_g$ is the scalar curvature and $k_g$ is the geodesic curvature, ${\rm dv}_g$ is the volume form on $\Sigma$ and $\lambda_g$ is the Riemannian measure of boundary induced by metric $g$. We also denote the interior of $\Sigma$ by $\Sigma^{\circ}.$\\
The quantization of the Liouville action is precisely LQFT: one wants to make sense of  the following measure on some appropriate functional space $\Sigma'$ made up of (generalized) functions $\varphi:\Sigma\to \R$
\begin{equation}\label{pathintegral}
F\mapsto \Pi_{\gamma,\mu} (g,F):=\int_{\Sigma'} F(\varphi)e^{-S_L(g,\varphi)}\,D\varphi
\end{equation}
where $D\varphi$ stands for the ``formal uniform measure'' on $\Sigma'$. Up to renormalizing this measure by its total mass, this formalism describes the law of some random (generalized) function $\varphi$ on $\Sigma$, which stands for the (log-)conformal factor of a random metric of the form $e^{\gamma \varphi}g$ on $\Sigma$.  In physics, LQFT is known to be a CFT with central charge ${\bf c_{\rm L}}:=1+6Q^2$  continuously ranging in $[25,\infty)$
for the particular values 
\begin{equation}
\gamma\in (0,2],\quad Q=\frac{2}{\gamma}+\frac{\gamma}{2}.
\end{equation}
 Of course, this description is purely formal, and giving a mathematical description of this picture is a longstanding problem, which goes back to the work of Polyakov \cite{Pol}.  The rigorous construction of such an object has been carried out recently in \cite{DKRV} in genus $0$, \cite{DRV} in genus $1$ and \cite{GRV} for all hyperbolic Riemann surfaces. When the Riemann surfaces have boundaries, \cite{HRV} give a rigorous construction of LCFT on the disk, and the annulus was treated in \cite{Remy}. The main result of this paper is to give a unified construction of LCFT for all Riemann surfaces with boundaries.
 Let us also mention that in \cite{DMS},  they use the mating of trees strategy and conformal welding to develop a theory of quantum wedges, quantum sphere and quantum cones. In the recent past, Morris Ang, Nina Holden, and Xin Sun gives a construction of quantum disk \cite{AHS20}, which is a finite-area and symmetric variant of quantum wedge. They find that the mating of trees strategy can be combined with the integrability of LCFT, reveal a series of deep results in SLE and LCFT, for example in \cite{ARS21}, they solve the bulk 1-point function in Liouville disk, which called FZZ formula; in \cite{AS}, which serves as one kind of basic structure constants in the boundary Liouville theory. Their techniques live on the simply connected topology, but the conformal welding theory should entwine with the Liouville theory in the higher genus topology.\\
{\bf Outline of this paper} we first notice that in every conformal class $ [g]:=\{ e^{\varphi}g; \varphi\in C^\infty(\Sigma)\}$, there is a unique special metric, which is called the type I uniform metric. In these metrics, we can easily use the doubling strategy to glue two same bordered surfaces as a closed surface, then we can use some estimation of Green function on a closed Riemann surface to construct our theory rigorous. For general metric, we first prove the Weyl anomaly formula \eqref{covconf}, then deduce the convergent property from the type I uniform metric case. The diffeomorphism invariance of the correlation function is an immediate consequence of our construction since our construction is coordinate-free \eqref{diffeoinv}. As an application of our construction, we prove certain Markov property of GFF on Riemann surface with boundaries.\\
{\bf In the forthcoming works}, the author will prove the convergence of the bosonic Liouville quantum gravity $\mathcal{Z}_{LQG}$ for hyperbolic type ($2(1-{\bf g}(\Sigma))-k < 0$) Riemann surface with boundaries. For the annulus case, the convergence of LQG was first shown in \cite{Remy} and later \cite{Wu} provided an explicit formula based on FZZ formula. This convergent issue is a basic question when we consider the LCFT couples with matter fields and ghost fields to get bosonic string theory. The construction of LCFT on Riemann surfaces with boundaries will also be a cornerstone for our range of work on Segal's axiom and bootstrap for Liouville theory in the boundary setting. The asymptotic of LCFT correlation function near the boundary of the moduli space is also directly related to the analyticity of conformal block on spectrum parameter $Q+iP$ under plumbing construction.
\subsubsection*{Acknowledgements:} The author would like to thank R\'emi Rhodes and Colin Guillarmou for their useful conversations. 

\subsubsection*{Notations:}
\begin{itemize}
\item We define $\Delta_g$, $d_g$ as the Laplacian and gradient on Riemannian manifold with respect to metric $g$, We define $\partial_{n_g}$ is the interior pointing vector field which is normal to $\partial\Sigma$;
\item We define $C^{\infty}_0(\Sigma)$ as the $C^{\infty}$ function on $\Sigma$ with compact support. We define $\|\cdot\|_g$ as the ordinary $L^2$ norm with respect to volume form ${\rm dv}_g$ and $\| f\|_{1,g}=\int_{\Sigma}|f(x)|^2{\rm dv}_g(x)+\int_{\Sigma}|df(x)|^2{\rm dv}_g(x)$. We denote the closure of $C^{\infty}_0(\Sigma)$ (resp. $C^{\infty}(\Sigma)$) under $\|\cdot\|_{1,g}$ by $H^1_0(\Sigma,g)$ (resp. $H^1(\Sigma,g)$).
\end{itemize}

\section{Uniformisation for Riemann surface with boundaries}
The determinant of Laplacian $\Delta$ on a compact Riemann surface $(\Sigma,g)$ was first introduced by Ray and Singer \cite{RS1}, \cite{RS2} and has become one of the central objects in geometric analysis, algebraic geometry, and string theory. It is defined by using the analytic continuation of the spectral zeta function
$$
\zeta(s)=\sum_{0<\lambda \in \operatorname{Spec}(\Delta)} \lambda^{-s}
$$
here $\operatorname{Spec}(\Delta)$ is the set of eigenvalues (when the spectrum is discrete otherwise we need some scattering theory. But in this paper, we only need to deal with the discrete spectrum case), where $\zeta(s)$ is apriori defined for $\Re(s)>>0$ but has a meromorphic extension to $\mathbb{C}$ with no pole
at $s = 0$. So we can define the height function $h(\Sigma,g)$ as
$$
h(\Sigma,g)=-\log \operatorname{det} \Delta=\zeta^{\prime}(0)
$$
Osgood, Phillips, and Sarnak first analyzed the extremal properties of the height
function $h$ in each conformal class of a surface and showed the uniformization
theorem \cite{OPS}. Namely, in each conformal class of Riemannian metrics on a
compact surface, if there is no boundary, there is a unique metric of constant
curvature.
If the Riemann surface $(\Sigma,g)$ has boundaries $\partial \Sigma$, they also proved that in every conformal class $ [g]:=\{ e^{\varphi}g; \varphi\in C^\infty(\Sigma)\}$ of metric $g$, there is a unique type I metric, such that the Gaussian curvature is constant, and the geodesic curvature is 0; there is also a unique type II metric, such that the Gaussian curvature is 0 and the geodesic curvature is constant.\\
The analyticity of spectral zeta function is well-known for closed Riemann surface based on heat kernel estimation. When a compact surface has regular boundaries with Dirichlet condition or Neumann condition, this is also known, see \cite{Alva} for details. 
\section{The Gaussian Free Fields with type I uniform metric \texorpdfstring{$g_0$}{g0}}
In the beginning of this section, we give some basic review about spectrum theory of Laplacian operator on the Riemannnian manifold with boundary (possibly empty).\\
The main ingredient for discreteness of the spectrum is the following proposition:
\begin{proposition}[Rellich Theorem]
Let $(\Sigma, g)$ be a smooth compact Riemannian manifold with boundary (possibly empty) (Riemann surfaces are indeed in this case). The inclusion maps
\begin{align*}
& (1)\left(H_{0}^{1}(\Sigma, g),\|\cdot\|_{1,g}\right) \rightarrow\left(L^{2}\left(\Sigma, \mathrm{v}_{g}\right),\|\cdot\|_{g}\right), \text { and }\\
& (2) \left(H^{1}(\Sigma, g),\|\cdot\|_{1,g}\right) \rightarrow\left(L^{2}\left(\Sigma, \mathrm{v}_{g}\right),\|\cdot\|_{g}\right)
\end{align*}
are compact: the image in $L^{2}$ of a bounded set in $H^{1}$ or $H_{0}^{1}$ is relatively compact.
\end{proposition}
\begin{remark}
 The above proposition remains true under weaker assumptions on the regularity of $\partial\Sigma$, however the second assertion might be false if $\partial\Sigma$ is too irregular.
\end{remark}
In this paper, we will study the following eigenvalue problems when $\Sigma$ is compact:\\
- Closed Problem:
$$
\Delta f=\lambda f \text { in } \Sigma ; \quad \partial \Sigma=\emptyset
$$
- Dirichlet Problem
$$
\Delta f=\lambda f \text { in } \Sigma ; \quad f=0 \text{ on } \partial\Sigma
$$
- Neumann Problem:
$$
\Delta f=\lambda f \text { in } \Sigma ;  \quad \partial_n f=0 \text{ on } \partial\Sigma
$$
When view $\Delta$ as an unbounded operator in $L^2(\Sigma)$, then each of the three operator with its domain is symmetric and positive and has a unique self-adjoint extension, we note their extensions as $\Delta$, $\Delta^D$ and $\Delta^N$ respectively. Then the spectrum theory says \cite[Theorem 18]{Be}:
\begin{proposition}\label{spectrum}
Let $\Sigma$ be a smooth compact manifold with boundary $\partial \Sigma$ (which may be empty), and consider one of the above mentioned eigenvalue problems. Then:\\
$(1)$. The set of eigenvalue consists of an infinite sequence $0<\lambda_{1} \leq \lambda_{2} \leq \lambda_{3} \leq \ldots \rightarrow \infty$, where 0 is not an eigenvalue in the Dirichlet problem;\\
$(2)$. Each eigenvalue has finite multiplicity and the eigenspaces corresponding to distinct eigenvalues are $L^{2}(\Sigma)$-orthogonal;\\
$(3)$. The direct sum of the eigenspaces $E\left(\lambda_{i}\right)$ is dense in $L^{2}(\Sigma)$ for the $L^{2}$ norm. Furthermore, each eigenfunction is $C^{\infty}$-smooth.
\end{proposition}
For any Riemannian metric $g$ on $\Sigma$, we can choose a type I uniform metric $g_0$ which is conformal to $g$. For consistency in writing, we first suppose the surface is hyperbolic type, i.e. $2(1-{\bf g}(\Sigma))-k < 0$ where $k$ is the number of connected components of boundaries. (the excluded case will be discussed in section \eqref{Diskannulus}). Now we consider the following Neumann problem for Green function $G^{\Sigma}_{g_0}(\cdot,\cdot)$:
\begin{equation}\label{neu_prob}  
  \left \lbrace  \begin{array}{lcl} \Delta^{\Sigma}_{g_0} G^{\Sigma}_{g_0}(x,y) =  \frac{\delta^{\Sigma}(x-y)}{g_0}- \frac{1}{\Vol_{g_0}(\Sigma)} & \text{for} & x\in \Sigma  \\ 
 \partial^{\Sigma}_{n_{g_0}}G^{\Sigma}_{g_0}(x,y) =0 & \text{for} &  x \in \partial \Sigma   \\ 
 \int_{\Sigma} G^{\Sigma}_{g_0}(x,y) {\rm dv}_{g_0}(y) = 0 & & \end{array}  \right. 
 \end{equation}
here $\frac{\delta^{\Sigma}}{g_0}$ has total mass 1 with respect to volume $g_0$ on $\Sigma$.\\
 Now we perform the {\bf doubling trick} to solve this problem, since the boundaries of $(\Sigma,g_0)$ are all geodesic curves. We may glue $\Sigma$ with its isometric involution $\sigma (\Sigma)$ along the boundary $\partial\Sigma$, we define the metric on $\sigma(\Sigma)$ as $\sigma(g_0)(\cdot):=g_0(\sigma(\cdot))$. Here $\sigma$ is the involution, i.e. $\sigma^2 = Id$, $\sigma|_{\partial \Sigma}=Id$, and we call the resulting compact surface without boundary $d\Sigma$. Locally, the near boundary part of $(\Sigma,g_0)$ is mapped to the left side of half upper plane $(\mathbb{H}\cap\{{Re(z)<0}\},g_{\mathbb{H}})$ and the corresponding part of $(\sigma(\Sigma),\sigma(g_0))$ to the right side of half upper plane $(\mathbb{H}\cap\{{Re(z)>0}\},g_{\mathbb{H}})$. Since $\mathbb{H}\cap \{Re(z)=0\}$ is a geodesic in the Poincare metric $g_{\mathbb{H}}$, so the metric extends smoothly to $d\Sigma$ and gives a smooth hyperbolic metric $\tilde{g_0}$ on $d\Sigma$. We first solve the Green function for compact Riemann surface $d\Sigma$:
 \begin{equation}\label{neu_prob}  
  \left \lbrace  \begin{array}{lcl} \Delta^{d\Sigma}_{\tilde{g_0}} G^{d\Sigma}_{\tilde{g_0}}(x,y) =  \frac{\delta^{d\Sigma}(x-y)}{\tilde{g_0}}- \frac{1}{\Vol_{\tilde{g_0}}(d\Sigma)} & \text{for} & x\in d\Sigma  \\ 
 \int_{\Sigma} G^{d\Sigma}_{\tilde{g_0}}(x,y) {\rm dv}_{\tilde{g_0}}(y) = 0 & & \end{array}  \right. 
 \end{equation}
 \begin{proposition}
 For every $x, y \in \Sigma$, we have $G^{\Sigma}_{g_0}(x,y)= G^{d\Sigma}_{\tilde{g_0}}(x,y)+G^{d\Sigma}_{\tilde{g_0}}(x,\sigma(y))$
 \end{proposition}\label{green}
 \begin{proof}
 First we notice that $G^{d\Sigma}_{\tilde{g_0}}(x,y)=G^{d\Sigma}_{\tilde{g_0}}(\sigma(x),\sigma(y))$ for every $x, y \in d\Sigma $, since $(\Delta_{g_0} f)\circ \sigma=\Delta_{{\sigma(g_0)}}(f\circ \sigma)$ and $\sigma(\tilde{g_0})(z)=\tilde{g_0}(\sigma(z))$, then we have
 \begin{align}
     \Delta_{g_0}(G^{d\Sigma}_{\tilde{g_0}}(x,y)+G^{d\Sigma}_{\tilde{g_0}}(x,\sigma(y))&=\frac{\delta^{d\Sigma}(x-y)}{\tilde{g_0}}- \frac{1}{\Vol_{\tilde{g_0}}(d\Sigma)}+\frac{\delta^{d\Sigma}(x-\sigma(y))}{\tilde{g_0}}- \frac{1}{\Vol_{\tilde{g_0}}(d\Sigma)}\\
     &= \frac{\delta^{\Sigma}(x-y)}{g_0}- \frac{1}{\Vol_{g_0}(\Sigma)} 
 \end{align}
 \begin{align}
     \partial_{n_{g_0}}\Big(G^{d\Sigma}_{\tilde{g_0}}(x,y)+G^{d\Sigma}_{\tilde{g_0}}(x,\sigma(y))\Big)= \partial_{n_{g_0}}\Big(G^{d\Sigma}_{\tilde{g_0}}(x,y)+G^{d\Sigma}_{\tilde{g_0}}(\sigma(x),y)\Big)=0
 \end{align}
 \begin{align}
     \int_{\Sigma} G^{\Sigma}_{g_0}(x,y) {\rm dv}_{g_0}(y)=\int_{\Sigma} (G^{d\Sigma}_{\tilde{g_0}}(x,y)+G^{d\Sigma}_{\tilde{g_0}}(x,\sigma(y)) {\rm dv}_{\tilde{g_0}}(y) =  \int_{M} G^{d\Sigma}_{\tilde{g_0}}(x,y) {\rm dv}_{\tilde{g_0}}(y) = 0
 \end{align}

 We also know that $G^{\Sigma}_{g_0}(x,y)$ is symmetric since $G^{d\Sigma}_{\tilde{g_0}}(x,y)$ is symmetric.
\end{proof}

For a general metric $g$ in the conformal class of $g_0$, i.e. $g=e^{\varphi}g_0$, we define the green function by the following relation\begin{equation} 
    G^{\Sigma}_{g}(x,y)=G^{\Sigma}_{g_0}(x,y)-m_g(G^{\Sigma}_{g_0}(x,\cdot))-m_g(G^{\Sigma}_{g_0}(\cdot,y))+m_g(G^{\Sigma}_{g_0}(\cdot,\cdot)).
\end{equation}

Here $m_g(f)$ means the average of $f$ on $\Sigma$ with respect to volume form $dv_g$, and $m_g(f(\cdot,\cdot))$ means double average in both coordinates. One may check easily that
\[\Delta_{g} G^{\Sigma}_{g}(x,y)=\Delta_{g} (G^{\Sigma}_{g_0}(x,y)-m_g(G^{\Sigma}_{g_0}(x,\cdot)))=\frac{\delta^{\Sigma}(x-y)}{g}- \frac{1}{\Vol_{g}(\Sigma)}\]
So it solves the following Neumann problem:
\begin{equation}\label{neu_prob}  
  \left \lbrace  \begin{array}{lcl} \Delta_{g} G^{\Sigma}_{g}(x,y) =  \frac{\delta^{\Sigma}(x-y)}{g}- \frac{1}{\Vol_{g}(\Sigma)} & \text{for} & x\in \Sigma  \\ 
 \partial_{n_{g}}G^{\Sigma}_{g}(x,y) =0 & \text{for} &  x \in \partial \Sigma   \\ 
 \int_{\Sigma} G^{\Sigma}_{g}(x,y) {\rm dv}_{g}(y) = 0 & & \end{array}  \right. 
 \end{equation}
 Now we can define the Gaussian free field for the Neumann condition (for a rigorous construction, see subsection \eqref{GFF}) on $(\Sigma,g_0)$ as $\mathbb{E}[X^{\Sigma}_{g_0}(x)X^{\Sigma}_{g_0}(y)]=2\pi G^{\Sigma}_{g_0}(x,y)$ and Gaussian free field on $(d\Sigma,\tilde{g_0})$ as $\mathbb{E}[X^{d\Sigma}_{\tilde{g_0}}(x)X^{d\Sigma}_{\tilde{g_0}}(y)]=2 \pi G^{d\Sigma}_{\tilde{g_0}}(x,y)$. From \eqref{green}, we know that $X^{\Sigma}_{g_0}(x)\stackrel{law}{=}\frac{X^{d\Sigma}_{\tilde{g_0}}(x)+X^{d\Sigma}_{\tilde{g_0}}(\sigma(x))}{
\sqrt{2}}$ in distribution.\\

 The Gaussian free field $X^{\Sigma}_{g}$ for a general metric $g$ which is in the conformal class of $g_0$ can be defined as $\mathbb{E}[X^{\Sigma}_{g}(x)X^{\Sigma}_{g}(y)]=2\pi G^{\Sigma}_{g}(x,y)$, then we have $X^{\Sigma}_{g}\stackrel{law}{=}X^{\Sigma}_{g_0}-m_g(X^{\Sigma}_{g_0})$ in distribution. From now on we always write the metric $\tilde{g_0}$ on $d\Sigma$ as $g_0$  if there is no ambiguity.\\
 \subsection{Construction of the GFF on \texorpdfstring{$\Sigma$}{}}\label{GFF}
 As the GFF is an infinite dimensional Gaussian, it is natural to expect a construction through its projections onto finite dimensional subspaces. Recall that the Neumann Laplacian $\Delta^N_g$ has an orthonormal basis of real valued eigenfunctions $(\varphi_j)_{j\in \N_0}$ in $L^2(\Sigma,g)$ with associated eigenvalues $\la_j\geq 0$; we set $\la_0=0$ and $\varphi_0=({\rm Vol}_g(\Sigma))^{-1/2}$. The Laplacian can thus be seen as a symmetric operator on an infinite dimensional space. Denote $H_n$  the finite dimensional space  spanned by the first $n$ eigenfunctions $(\varphi_j)_{j=1,\dots,n}$ of the Laplacian. Notice that for $\varphi=\sum_{j=1}^n \tilde{\alpha}_j \varphi_j$ we have
$\|d\varphi\|^2_{L^2}=\langle \Delta\varphi,\varphi\rangle_g= \sum_{j=1}^n\tilde{\alpha}_j^2 \lambda_j$, here we use the {\bf Neumann boundary condition} and Green formula. Therefore the projection to $H_n$ of the formal measure $e^{-\frac{1}{4\pi}||d\varphi||_{L^2}^2}D\varphi$ is naturally understood as
\begin{align*}
 \int_{H_n} F(\varphi)e^{-\frac{1}{4\pi}\|d\varphi\|^2_{L^2}}D\varphi & =\int_{\R^n}F\big(\sum_{j=1}^n \tilde{\alpha}_j \varphi_j\big)\prod_{j=1}^n\Big(e^{-\frac{1}{4\pi} (\tilde{\alpha}_j)^2 \lambda_j} d \tilde{\alpha}_j\Big)  \\
 & =  (2\pi)^{n/2}\Big(\prod_{j=1}^{n}\lambda_j\Big)^{-1/2} \int_{\R^n}F\big(  \sqrt{2 \pi} \sum_{j=1}^n\alpha_j \frac{\varphi_j }{\sqrt{\lambda_j} } \big)\prod_{j=1}^n\Big(e^{-\frac{\alpha_j^2}{2}}  d\alpha_j\Big) \\
\end{align*}
for appropriate bounded measurable functionals $F$. The mass of this Gaussian measure is $(2\pi)^{n}\Big(\prod_{j=1}^{n}\lambda_j\Big)^{-1/2}$.  
Renormalized by its mass, this measure becomes a probability measure describing the law  of the random function
\begin{equation}\label{truncGFF}
X_n:=  \sqrt{2 \pi} \sum_{j=1}^n\alpha_j \frac{\varphi_j}{\sqrt{\la_j}}
\end{equation}
 where $(\alpha_j)_j$ are independent Gaussian random variables with law $\mc{N}(0,1)$. 

To obtain the description of the GFF, one has to take the limit $n\to\infty$. 
\begin{proposition}
 The sum \eqref{truncGFF} converges almost surely in the Sobolev space $H^{-s}(\Sigma)$ for each $s>0$
\end{proposition} 
\begin{proof}
This is a consequence of standard Weyl's asymptotic for the Neumann eigenvalues on a smooth and bounded domain. If we order $\lambda_n$ increasingly, then $\lambda_n\sim n $, see \cite[chapter 3 page 70]{Be}.
\end{proof}
The mass $(2\pi)^{n}\Big(\prod_{j=1}^{n}\lambda_j\Big)^{-1/2}$ diverges  as $n\to \infty$ but this infinite product can be analytic continued as $ ( \det '(\tfrac{1}{2\pi}\Delta_g)) ^{-1/2}=\zeta'(0)$ by spectrum Zeta function $\zeta(s)=\sum_{0<\lambda \in \operatorname{Spec}(\frac{1}{2\pi}\Delta)} \lambda^{-s}$.

 \begin{remark}
If we hope to get a Gaussian free field with the Dirichlet $0$ boundary condition on $\partial \Sigma$, the Green function of this boundary problem satisfies:\\ 
  For every $x, y \in \Sigma$, we have $G^{\Sigma, D}_{g_0}(x,y)= G^{d\Sigma}_{g_0}(x,y)-G^{d\Sigma}_{g_0}(x,\sigma(y)).$ \\And then we can define $X^{\Sigma, D}_{g_0}(x)$ as $\mathbb{E}[X^{\Sigma, D}_{g_0}(x)X^{\Sigma, D}_{g_0}(y)]=2\pi G^{\Sigma, D}_{g_0}(x,y)$. What's more, the GFF with Dirichlet $0$ boundary condition has the same law in the same conformal class, i.e. $X^{\Sigma, D}_{e^{\varphi}g_0}(x)\stackrel{law}{=}X^{\Sigma, D}_{g_0}(x)$.
 \end{remark}
\section{Gaussian Multiplicative Chaos}
Since $X^{\Sigma}_{g}$ lives almost surely in the space of distribution $H^{-s}(M)$ and when we hope to define its exponential, i.e. the Gaussian multiplicative chaos, we need to introduce the regularization procedure. We define $X^{\Sigma}_{g,\epsilon}(x)$ as the average of $X^{\Sigma}_{g}$ on the intersection of geodesic circle $(C_{\epsilon}(x),g)$ with $(\Sigma,g)$ . First we focus on the type I uniform metric $g_0$. We define the bulk measure as $M_{\gamma, g_0, \epsilon}(dx) = \epsilon^{\frac{\gamma^2}{2}} e^{ \gamma X^{\Sigma}_{g_0, \epsilon} (x) } {\rm dv}_{g_0}(x)$ and boundary measure as $M^{\partial}_{\gamma, g_0, \epsilon}(dx) = \epsilon^{\frac{\gamma^2}{4}} e^{ \frac{\gamma}{2} X^{\Sigma}_{g_0, \epsilon} (x) } d\lambda_{g_0}(x)$. To deal with the convergence issue of this measure, we first prove it's true for uniform type I metric, and other metrics follows from the conformal change \eqref{Measurechange}. To do this, we recall the following estimate of Green function for Riemann surfaces without boundary from \cite{GRV}.
\begin{lemma}\label{Greenesti}
For $X^{d\Sigma}_{g_0, \epsilon}$, we always have as $\epsilon\to 0$, $\mathbb{E}[(X^{d\Sigma}_{g_0, \epsilon})^2]=-\log(\epsilon)+W_{g_0}(x)+o(1)$, where $W_{g_0}(x)$ is a smooth function on $d\Sigma$. And near each point $x_0\in \Sigma$, there are coordinates $z\in B(0,1-\delta)\subset \cc$ so that $g_P=4|dz|^2/(1-|z|^2)^2$ and near $x_0$, $G^{d\Sigma}_{g_P}(z,z')= -\frac{1}{2\pi}\log |z-z'|+F(z,z')$
with $F$ smooth. 
\end{lemma}

\begin{proposition}
As $\epsilon \to 0$,  the limit of the bulk measure  $M_{\gamma, g_0, \epsilon}(dx) = \epsilon^{\frac{\gamma^2}{2}} e^{ \gamma X^{\Sigma}_{g_0, \epsilon} (x) } {\rm dv}_{g_0}(x)$ and the boundary measure $M^{\partial}_{\gamma, g_0, \epsilon}(dx) = \epsilon^{\frac{\gamma^2}{4}} e^{ \frac{\gamma}{2} X^{\Sigma}_{g_0, \epsilon} (x) } d\lambda_{g_0}(x)$ converge in probability and weakly in the space of Radon measures towards a random measure $M_{\gamma, g_0}$ and $M^{\partial}_{\gamma, g_0}$ respectively. And the resulting measures don't depend on the regularization procedure.
\end{proposition}

\begin{proof}
Since when $s\in \partial \Sigma$, $X^{\Sigma}_{g_0,\epsilon}(s)=\sqrt{2}X^{d\Sigma}_{g_0,\epsilon}(s)$, so the boundary measure can be expressed as 
$$e^{\frac{\gamma}{2} X^{\Sigma}_{g_0,\eps}(s)-\frac{\gamma^2}{8}\E[X^{\Sigma}_{g_0,\eps}(s)^2]}e^{\frac{\gamma^2}{4}W_{g_0}(s)}{\rm d\lambda}_{g_0}(s)
.$$ We may write $e^{\frac{\gamma^2}{4}W_{g_0}(s)}{\rm d\lambda}_{g_0}(s)$ as a bounded measure $d\rho(s)$, it's easy to see when $\gamma\in(0,\sqrt{2})$, we can show the $L^2$ convergence. We then need to focus on the bulk measure. By previous lemma, it's suffice to study the convergence of the measures
\[e^{\gamma X^{\Sigma}_{g_0,\eps}(x)-\frac{\gamma^2}{2}\E[X^{\Sigma}_{g_0,\eps}(x)^2]}e^{\frac{\gamma^2}{4}(W_{g_0}(x)+W_{g_0}(\sigma(x))+4\pi G^{d\Sigma}_{g_0}(x,\sigma(x)))}{\rm dv}_{g_0}(x).\]
Note that the volume form ${\rm dv}_{g_0}$ in bounded in $\Sigma$, since $g_0$ can be extended smoothly to the whole $dM$. So we can write $e^{\frac{\gamma^2}{4}(W_{g_0}(x)+W_{g_0}(\sigma(x)))}{\rm dv}_{g_0}(x)$ as some bounded measure ${\rm d\sigma}(x)$.\\
The map from Poincare half upper plane to Poincare disk sends $(x,\sigma(x))$ to $(z,\Bar{z})$ ( recall $x$ and $\sigma(x) $ are symmetric with respect to the imaginary axis). Now we choose the Poincare metric on disk $\mathbb{D}$, i.e. $g_p=4|dz|^2/(1-|z|^2)^2$, we have $4\pi  G^{d\Sigma}_{g_p}(z,\Bar{z})=-2 \log |\Im(z)|+F(z)$ with $F(z)$ smooth. 
Then we can prove the $L^2$ convergence of $M_{\gamma, g_0, \epsilon}(A)$ and $M_{\gamma, g_0, \epsilon}(A)$ is a Cauchy sequence for $A\in M^{\circ}$ when $\gamma^2<2$ . When $\gamma \in [\sqrt{2},2)$, the estimate becomes more complicated and we refer to \cite{Ber} for a simple proof.
\end{proof}
For $\gamma=2$, both bulk measure and boundary measure are 0, so we need extra $\sqrt{\ln(\frac{1}{\epsilon})}$ to renormalize them, in this case. the limiting measures are exist and nontrivial, the proof strategy is similar with \cite[section 4]{HRV} .
\begin{proposition}
When $\gamma=2$, as $\epsilon \to 0$,  the limit of the bulk measure  $M_{2, g_0, \epsilon}(dx) = \sqrt{\ln(\frac{1}{\epsilon})}\epsilon^{2} e^{ 2 X^{\Sigma}_{g_0, \epsilon} (x) } {\rm dv}_{g_0}(x)$ and the boundary measure $M^{\partial}_{2, g_0, \epsilon}(dx) = \sqrt{\ln(\frac{1}{\epsilon})}\epsilon^{1} e^{  X^{\Sigma}_{g_0, \epsilon} (x) } d\lambda_{g_0}(x)$ converge in probability and weakly in the space of Radon measures towards a random measure $M_{2, g_0}$ and $M^{\partial}_{2, g_0}$ respectively. And the resulting measures don't depend on the regularization procedure.
\end{proposition}
Due to the estimate of Green function $G^{d\Sigma}_{g_0}(z,z')$, as the strategy  in \cite{HRV}, one can also show that:
\begin{proposition}\label{finiteness}
$M_{\gamma, g_0}(\Sigma)$ and $M^{\partial}_{\gamma, g_0}(\partial \Sigma)$ are almost surely finite.
\end{proposition}
\begin{proof}
We only prove this when $\gamma<2$, the critical $\gamma=2$ case can be treated similarly as in \cite{HRV}.\\
First we notice that, for the boundary measure, 
\[M^{\partial}_{\gamma, g_0, \epsilon}(dx) = \epsilon^{\frac{\gamma^2}{4}} e^{ \frac{\gamma}{2} X^{\Sigma}_{g_0, \epsilon} (x) } d\lambda_{g_0}(x)=e^{\frac{\gamma}{2} X^{\Sigma}_{g_0,\eps}(x)-\frac{\gamma^2}{8}\E[X^{\Sigma}_{g_0,\eps}(x)^2]}e^{\frac{\gamma^2}{4} W_{g_0}(x)}{\rm d\lambda}_{g_0}(x).\]
So $\mathbb{E}[M^{\partial}_{\gamma, g_0}(\partial \Sigma)]=\int_{\partial M}e^{\frac{\gamma^2}{4} W_{g_0}(x)}{\rm d\lambda}_{g_0}(x) < \infty$. When we focus the bulk measure, we can also show easily that $\mathbb{E}[M_{\gamma, g_0}(\Sigma)]$ is finite when $\gamma^2 <2$. \\
When $\gamma^2 \geq 2$, we prove that $\mathbb{E}\Big[\Big(M_{\gamma, g_0}(\Sigma)\Big)^{\alpha}\Big]$ is finite for some positive $\alpha$. As the local argument in previous proposition, we only need to show that 
\[\mathbb{E}[\Big(\int_{\mathbb{H}\cap\mathbb{D}}e^{\gamma X^{\Sigma}_{g_0,\eps}(z)-\frac{\gamma^2}{2}\E[X^{\Sigma}_{g_0,\eps}(z)^2]}(\frac{1}{|\Im (z)|}\Big)^{\frac{\gamma^2}{2}}d(z)\Big)^{\alpha}]< \infty\]
This is true when $\gamma >1$ and $\alpha <\frac{1}{\gamma}$, see section 3.1 in \cite{HRV}.
\end{proof}
\begin{proposition}\label{Measurechange}
For a metric $g=e^{\varphi}g_0$, we have the metric change formula for GMC:
\[M_{\gamma, g}(dx)\stackrel{law}{=} e^{\gamma(\frac{Q}{2}\varphi-m_g(X^{\Sigma}_{g_0}))}M_{\gamma, g_0}(dx)\] and
\[M^{\partial}_{\gamma, g}(dx)\stackrel{law}{=} e^{\frac{\gamma}{2}(\frac{Q}{2}\varphi-m_g(X^{\Sigma}_{g_0}))}M^{\partial}_{\gamma, g_0}(dx)\] 
\end{proposition}
\begin{proof}
Define $\tilde{X^{\Sigma}_{g_0,\epsilon}}(x)$ as the average of $X^{\Sigma}_{g_0}$ on the intersection of geodesic circle $(C_{\epsilon}(x),g)$ with $(\Sigma,g)$. We need to understand the metric change in the Green function $G^{\Sigma}_{g_0}(x+e^{\frac{\omega(x)}{2}}e^{i\theta},x+e^{\frac{\omega(x)}{2}}e^{i\theta'})$. For $x\in\Sigma^{\circ}$  we simply notice that $|(x+e^{\frac{\omega(x)}{2}}e^{i\theta})-\sigma(x+e^{\frac{\omega(x)}{2}}e^{i\theta'})|$ won't contribute $\omega$ term, then we can use the strategy of proposition 3.4 in \cite{GRV} show that 
\[\mathbb{E}[\tilde{X^{\Sigma}_{g_0,\epsilon}}(x)]^2-\mathbb{E}[X^{\Sigma}_{g_0,\epsilon}(x)]^2=\frac{1}{2}\varphi(x) +o(1).\] Then we use $X^{\Sigma}_{g,\epsilon}(x)=\tilde{X^{\Sigma}_{g_0,\epsilon}}(x)-m_g(X^{\Sigma}_{g_0})$ to conclude the measure change formula.
For the case $s\in\partial\Sigma$, we use $X^\sigma_{g_0}=\sqrt{2}X^{d\sigma}_{g_0}$.
\end{proof}
We may also define $\tilde{M_{\gamma, g_0}}(dx)$ as the bulk GMC measure corresponds to $\tilde{X^{\Sigma}_{g_0,\epsilon}}(x)$, then we have $\tilde{M_{\gamma, g_0}}(dx)=e^{\gamma\frac{Q}{2}\varphi(x)}M_{\gamma, g_0}(dx)$, and similar thing holds for the boundary GMC measure.

\section{The Partition function of the LCFT.}
We define the free field partition function correspond to Neumann boundary condition \[Z^N_{GFF}(\Sigma,g):=({\det}'(\Delta^N_{g})/{\rm Vol}_{g}(\Sigma))^{-1/2} \exp{(-\frac{1}{8\pi}}\int_{\partial \Sigma} k_g {\rm d\lambda}_g)\]
Here we contain an extra term of integration of geodesic curvature along boundaries $\partial\Sigma$ to produce the right conformal anomaly formula, this definition is different from the free field partition function in \cite{HRV} and \cite{Remy}.
Now we have a Polyakov formula for Riemann surface with Neumann boundary. The proof is based on estimating the heat kernel of the Laplacian operator, for details, see \cite{Alva}
\begin{proposition}\label{Polya}
Suppose ${g}=e^{\varphi}g_0$, then
\[ \log\Big(\frac{{\det}'(\Delta^N_{g})/{\rm Vol}_{g}(\Sigma) \exp(-\frac{1}{4\pi}\int_{\partial \Sigma} k_g {\rm d\lambda}_g)}{{\det}'(\Delta^N_{g_0})/{\rm Vol}_{g_0}(\Sigma) \exp{(-\frac{1}{4\pi}\int_{\partial \Sigma} k_{g_0} {\rm d\lambda}_{g_0})}}\Big)=\frac{1}{48\pi}\Big(\int_{\Sigma}(||d\varphi||^2_{g_0}+2R_{g_0}\varphi) {\rm dv}_{g_0}+4 \int_{\partial \Sigma} k_{g_0}\varphi d\lambda_{g_0}\Big) \]
\end{proposition}
\begin{proof}
We have the short time ($t$ small) behavior of heat trace
\begin{align}
\operatorname{Tr}[\varphi \exp (-t \Delta_g^N)] \sim & \frac{1}{4 \pi t} \int_{\Sigma}  \varphi\mathrm{d v}_g(z)+\frac{1}{8 \sqrt{\pi t}} \int_{\partial \Sigma}  \varphi\mathrm{d \lambda}_g(s)+\frac{1}{8 \pi} \int_{\partial \Sigma} \partial_{n_g} \varphi \mathrm{d\lambda}_g(s) \nonumber\\
&+\frac{1}{12 \pi}\left[\int_{\partial \Sigma}k_g \varphi \mathrm{d\lambda}_g(s) +\frac{1}{2} \int_{\Sigma}  R_g \varphi{\rm dv}_g(z)\right]+\mathrm{O}(\sqrt{t})
\end{align}
\begin{align}
\left.\frac{d}{d \varepsilon}\right|_{\varepsilon=0} \operatorname{Tr}\left(e^{-t \Delta_{e^{\epsilon\varphi}g_0}}\right)=-2 t \cdot \operatorname{Tr}\left(\varphi \Delta^N_{g_0} e^{t \Delta^N_{g_0}}\right)=\frac{d}{dt}\operatorname{Tr}\left(\varphi e^{-t \Delta^N_{g_0}}\right).
\end{align}
Then by similar method in \cite[Theorem 3.1]{Chang} we get the result.
\end{proof}
so as a corollary, we have the conformal anomaly formula for $Z^N_{GFF}(\Sigma,g)$, 
\begin{corollary}
\begin{equation}
    \frac{Z^N_{GFF}(\Sigma,g)}{Z^N_{GFF}(\Sigma,g_0)}=\exp(\frac{1}{48\pi}\Big(\int_{\Sigma}(||d\varphi||^2_{g_0}+2R_{g_0}\varphi) {\rm dv}_{g_0}+4 \int_{\partial \Sigma} k_{g_0}\varphi d\lambda_{g_0}\Big))
\end{equation}

\end{corollary}Similarly, if we define the Dirichlet boundary free field partition function 
\[Z^D_{GFF}(\Sigma,g):={\det}(\Delta^D_{g})^{-1/2} \exp{(+\frac{1}{8\pi}}\int_{\partial \Sigma} k_g {\rm d\lambda}_g)\]
we have
\begin{corollary}
\begin{equation}
    \frac{Z^D_{GFF}(\Sigma,g)}{Z^D_{GFF}(\Sigma,g_0)}=\exp(\frac{1}{48\pi}\Big(\int_{\Sigma}(||d\varphi||^2_{g_0}+2R_{g_0}\varphi) {\rm dv}_{g_0}+4 \int_{\partial \Sigma} k_{g_0}\varphi d\lambda_{g_0}\Big)).
\end{equation}

\end{corollary}

Now we start to define the partition function for LCFT. Let $g\in{\rm Met}(\Sigma)$ be a fixed metric on $\Sigma$. The mathematical definition of  the  LQFT measure is the following. Fix  $\gamma\in(0,2]$. For $F:  H^{-s}(\Sigma)\to\R$ (with $s>0$) a bounded continuous functional, set  
\begin{align}\label{partLQFT}
 &\Pi_{\gamma, \mu}(g,F):= Z^N_{GFF}(\Sigma,g) \\
 &\times \int_\R  \E\Big[ F( c+  X_{g}) \exp\Big( -\frac{Q}{4\pi}\int_{\Sigma}R_{g}(c+ X_{g} )\,{\rm dv}_{g} -\frac{Q}{2\pi}\int_{\partial \Sigma}k_{g}(c+ X_{g} )d\lambda_g- \mu  e^{\gamma c}M_{\gamma, g}(\Sigma)-\mu_{\partial} e^{\frac{\gamma}{2}c}M^{\partial}_{\gamma, g}(\partial \Sigma) \Big) \Big]\,dc .\nonumber
\end{align}
This quantity, if it is finite, gives a mathematical sense to the formal integral 
\[  \int F(\varphi)e^{-S_L(g,\varphi)}D\varphi\]
where $S_L(g,\varphi)$ is the Liouville action \eqref{Laction}. The partition function is the total mass of this measure, i.e $\Pi_{\gamma, \mu}(g,1)$. 

\begin{proposition}\label{totalmass}
For $g\in  {\rm Met}(\Sigma)$ and  $\gamma\in(0,2]$, we have $0<\Pi_{\gamma, \mu}(g,1)<+\infty$ and 
the mapping $$F\in C_b(H^{-s}(\Sigma),\R)\mapsto  \Pi_{\gamma, \mu}(g,F)$$ defines a positive finite measure.
\end{proposition} 
\begin{proof} First, we start from the sub-critical case $\gamma <2$. Since $F$ is bounded on M, so we only need to deal with when $\chi(\Sigma)<0$, let $g_0$ be a uniform type I metric in the conformal class of $g$.
\begin{align}
&\int_\R  e^{-Qc\chi(\Sigma)} \E\Big[ \exp\Big(- \mu  e^{\gamma c}M_{\gamma, g_0}(\Sigma)-\mu_{\partial} e^{\frac{\gamma}{2}c}M^{\partial}_{\gamma, g_0}(\partial \Sigma) \Big) \Big]\,dc\\
\leq &\min\left\{\int_\R  e^{-Qc\chi(\Sigma)} \E\Big[ \exp\Big(- \mu  e^{\gamma c}M_{\gamma, g_0}(\Sigma) \Big) \Big]\,dc, \int_\R  e^{-Qc\chi(\Sigma)} \E\Big[ \exp\Big(-\mu_{\partial} e^{\frac{\gamma}{2}c}M^{\partial}_{\gamma, g_0}(\partial \Sigma) \Big) \Big]\,dc\right\}\\
=& \min \left\{ \E \Big[M_{\gamma, g_0}(\Sigma)^{\frac{Q\chi(\Sigma)}{\gamma}}\Big], \E \Big[ M^{\partial}_{\gamma, g_0}(\partial \Sigma)^{\frac{2Q\chi(\Sigma)}{\gamma}}\Big]\right\}
\end{align}
Which is finite since both GMC have negative moments, see for example this property, \cite[Proposition 3.6]{RoVa} . 
For the critical case $\gamma=2$, both GMC have negative moments which is guaranteed by \cite[Corollary 14]{Rnew12}. \\
To prove the convergence of partition function in a general metric $g\in [g_0]$, we can simply use the conformal anomaly for the LCFT as the following.

\end{proof}

\begin{proposition}\label{covconf}{\bf (Conformal anomaly)}
Let $Q=\frac{\gamma}{2}+\frac{2}{\gamma}$ with  $\gamma\in(0,2]$  and $g_0$ be a type I uniform metric on $\Sigma$.
The partition function satisfies the following conformal anomaly:  
if ${g}=e^{\varphi}g_0$ for some $\varphi\in C^\infty(\Sigma)$, we have
\[ \Pi_{\gamma, \mu}(g,F)=\Pi_{\gamma, \mu}(g_0,F(\cdot\,-\tfrac{Q}{2}\varphi))\exp\Big(\frac{1+6Q^2}{96\pi}\Big(\int_{\Sigma}(||d\varphi||^2_{g_0}+2R_{g_0}\varphi) {\rm dv}_{g_0}+4 \int_{\partial \Sigma} k_{g_0}\varphi d\lambda_{g_0}\Big)\Big).\]
\end{proposition}

\begin{proof}  We first shift $X_{g}$ to $X_{g_0}$ in the expression defining 
$\Pi_{\gamma, \mu}(g,F)$ and by using \eqref{Measurechange}, are thus left with considering the following quantity, note as $A$ for simplify.
\[\int_\R  \E\Big[ F( c+  X_{g_0})  \exp\Big( -\frac{Q}{4\pi}\int_{\Sigma}R_{g}(c+ X_{g_0}  )\,{\rm dv}_{g} -\frac{Q}{2\pi}\int_{\partial \Sigma}k_{g}(c+ X_{g_0} )d\lambda_g-\mu e^{\gamma c} e^{\gamma\frac{Q}{2}\varphi}M_{\gamma, g_0}(\Sigma)-\mu_{\partial}    e^{\frac{\gamma}{2} c} e^{\frac{\gamma}{2}\frac{Q}{2}\varphi}M^{\partial}_{\gamma, g_0}(\partial \Sigma)\Big) \Big]\,dc.\]
Define the Gaussian random variable 
\[ Y:=-\frac{Q}{4\pi}\int_{\Sigma}R_{g} X_{g_0} \,{\rm dv}_{g} -\frac{Q}{2\pi}\int_{\partial \Sigma}k_{g} X_{g_0} d \lambda_{g}\] 
Since for every smooth function $f$ on $\Sigma$,
 \begin{equation}\label{gr_g}
  \int_{M} G^{\Sigma}_{g_0}(z,z') \Delta_{g_0} f(z') {\rm dv}_{g_0}(z') - \int_{\partial \Sigma} G^{\Sigma}_{g_0}(z,z') \partial_{n_{g_0}}f(z') d\lambda_{g_0}(z')  = -(f(z) - m_{g_0} (f) ). 
  \end{equation}
Then we have
\begin{equation}
\demi \E[Y^2]= \frac{Q^2}{16\pi}\int_{\Sigma}|d\varphi|_g^2{\rm dv}_g, \quad \E[YX_{g_0}]= -\frac{Q}{2}(\varphi-m_{g_0}(\varphi)).
\end{equation}
Therefore by applying Girsanov transform to the random variable $Y$, we see that $A$ equals
\[\int_\R  e^{\frac{Q^2}{16\pi}||d\varphi||^2_{g_0}-Qc\chi(\Sigma)}\E\Big[ F( c+  X_{g_0} -\tfrac{Q}{2}\varphi+\tfrac{Q}{2}m_{g_0}(\varphi)) \exp\Big( - \mu  e^{\gamma (c+\tfrac{Q}{2}m_{g_0}(\varphi))}M_{\gamma, g_0}(\Sigma)- \mu_{\partial}  e^{\frac{\gamma}{2} (c+\tfrac{Q}{2}m_{g_0}(\varphi))}M^{\partial}_{\gamma, g_0}(\partial \Sigma)\Big) \Big]\,dc.\]
It remains to make the change of variable $c\to c-\frac{Q}{2}m_{g_0}(\varphi)$ and we deduce that 
\[\int_\R  e^{\frac{Q^2}{16\pi}||d\varphi||^2_{g_0}-Q c\chi(\Sigma)+\demi Q^2\chi(\Sigma)m_{g_0}(\varphi)}\E\Big[ F( c+  X_{g_0}-\frac{Q}{2}\varphi) \exp\Big( - \mu  e^{\gamma c}M_{\gamma, g_0}(\Sigma)- \mu_{\partial}  e^{\frac{\gamma}{2} c}M^{\partial}_{\gamma, g_0}(\partial \Sigma) \Big) \Big]\,dc.\]
Since $R_{g_0}=-2$, $k_{g_0}=0$ and ${\rm Vol}_g(\Sigma)=-2\pi\chi(\Sigma)$ we have 
\[-\frac{Q}{4\pi}\int_{\Sigma}R_{g_0}(c+ X_{g_0} )\,{\rm dv}_{g_0}-\frac{Q}{2\pi}\int_{\partial \Sigma}k_{g_0}(c+ X_{g_0} )d\lambda_{g_0}=-Qc\chi(\Sigma)\]
\[  
m_{g_0}(\varphi)\chi(\Sigma)=\frac{1}{4\pi}\int_{\Sigma} R_{g_0}\varphi\, {\rm dv}_{g_0}+\frac{1}{2\pi}\int_{\partial \Sigma}k_{g_0}\varphi d\lambda_{g_0}\]
which shows that $A=\Pi_{\gamma, \mu}(g_0,F(\cdot\,-\tfrac{Q}{2}\varphi))Z_{GFF}(\Sigma,g)e^{\frac{6Q^2}{96\pi}\Big(\int_{\Sigma}(|d\varphi|^2_{g_0}+2R_{g_0}\varphi) {\rm dv}_{g_0}+4 \int_{\partial \Sigma} k_{g_0}\varphi d\lambda_{g_0}\Big)}$. \\
Finally, we use \eqref{Polya} to get the factor $1+6Q^2$.
\end{proof}

The constant ${\bf c}_L:=1+6Q^2$ describing the conformal anomaly is called the \emph{central charge} of the Liouville Theory.
Since all the objects in the construction of the Gaussian Free Field and the Gaussian multiplicative chaos are geometric (defined in a natural way from the metric), it is direct to get the following diffeomorphism invariance:
\begin{proposition}\label{diffeoinv} {\bf (Diffeomorphism invariance)}
Let  $g\in {\rm Met}(\Sigma)$  be a metric on $\Sigma$ and let $\psi:\Sigma\to \Sigma$ be an orientation preserving diffeomorphism which sends boundary to boundary. Then we have for each bounded measurable $F:H^{-s}(\Sigma)\to \R$ with $s>0$
\[ \Pi_{\gamma, \mu} (\psi^*g ,F)= \Pi_{\gamma, \mu}(g,F(\cdot \circ \psi)) .\]
\end{proposition}

\begin{proof} This follows directly from the fact that Laplacian dose not depend on the choice of coordinates, i.e. $(\Delta_g f)\circ \psi=\Delta_{\psi^*g}(f\circ \psi)$ more precisely, it follows from the identities
$$G^{\Sigma}_{\psi^*g}(x,y)=G^{\Sigma}_g(\psi(x),\psi(y)),\quad R_{\psi^*g}(x)=R_g(\psi(x)),\quad k_{\psi^*g}(x)=k_g(\psi(x)),\quad X^{\Sigma}_{\psi^*g} \stackrel{law}{=}X^{\Sigma}_g\circ\psi,$$
which are standard.
\end{proof}
\begin{remark}
We remark here that the LCFT with boundary is easier to be convergent since there are two negative exponential potentials, but it's hard to directly get the exact formulae by the BPZ equation and analytic strategy. 
\end{remark}
\section{Correlation function}
The correlation functions of  LCFT can be thought of as the exponential moments  $e^{\alpha \varphi(x)}$ of the random function $\varphi$, the law of which is ruled by the path integral \eqref{pathintegral}.  Yet, the field $\varphi$ is not a well-defined function as it belongs to $H^{-s}(M)$ for $s>0$, so the construction requires some care.  
As before let $g\in {\rm Met}(\Sigma)$. We fix $n$ points $z_1,\dots,z_n$  on $\Sigma^{\circ}$ with respective associated weights $\alpha_1,\dots,\alpha_n\in\R$. We denote ${\bf x}=(x_1,\dots,x_n)$ and ${\boldsymbol \alpha} =(\alpha_1,\dots,\alpha_n)$. We also fix $m$ points $s_1,\dots,s_m$ on $\partial\Sigma$ with respective associated weights $\beta_1,\dots,\beta_m\in\R$. We denote ${\bf s}=(s_1,\dots,s_n)$ and ${\boldsymbol \beta} =(\beta_1,\dots,\beta_m)$. 
The rigorous definition of the primary fields will require a regularization scheme. We introduce the following $\eps$-regularized functional
\begin{align}\label{actioninsertion}
 \Pi_{\gamma, \mu,\mu_\partial}^{{\bf z},{\boldsymbol  \alpha},{\bf s},{\boldsymbol  \beta}}  (g,F,\eps):=&Z^N_{GFF}(\Sigma,g)
  \int_\R  \E\Big[ F(c+  X_{g} ) \big(\prod_i V^{\alpha_i}_{g,\eps}(z_i)\prod_j V^{\frac{\beta_j}{2}}_{g,\epsilon}(s_j)\big) \times  \\  &\exp\Big( -\frac{Q}{4\pi}\int_{\Sigma}R_{g}(c+ X_{g} )\,{\rm dv}_{g} -\frac{Q}{2\pi}\int_{\partial \Sigma}k_{g}(c+ X_{g} )d\lambda_g- \mu  e^{\gamma c}M_{\gamma, g}(\Sigma)-\mu_{\partial} e^{\frac{\gamma}{2}c}M^{\partial}_{\gamma, g}(\partial \Sigma) \Big) \Big]\,dc .
 \nonumber
\end{align}
where we have set, 
\begin{equation*}
V^{\alpha}_{g,\eps}(z)=\eps^{\alpha^2/2}  e^{\alpha  (c+X_{g,\epsilon}(z)) } \quad and \quad V^{\frac{\beta}{2}}_{g,\eps}(s)=\eps^{\beta^2/4}  e^{\frac{\beta}{2} (c+X_{g,\epsilon}(s)) } 
\end{equation*}
Here the regularization is the same as the definition of GMC measure. Such quantities are called {\rm vertex operators}. Before stating the following proposition, we first recall the Seiberg bound,
\begin{align}
 & \sum_{i}\alpha_i +\sum_j\frac{\beta_j}{2} >Q\chi(\Sigma),\label{seiberg1}\\ 
 &\forall i,\quad \alpha_i<Q \label{seiberg2}\\& \forall j,\quad \beta_j<Q.\label{seiberg3}
 \end{align}
\begin{proposition}\label{correlation}
For all bounded continuous functionals $F: h \in  H^{-s}(\Sigma) \to F(h) \in\R$ with $s>0$, the limit
 \begin{equation*}
 \Pi_{\gamma, \mu,\mu_\partial}^{{\bf z},{\boldsymbol  \alpha},{\bf s},{\boldsymbol  \beta}}  (g,F) := \lim_{\eps\to 0}  \Pi_{\gamma, \mu,\mu_\partial}^{{\bf z},{\boldsymbol  \alpha},{\bf s},{\boldsymbol  \beta}}  (g,F,\eps) ,
\end{equation*}
 exists and is finite with $\Pi_{\gamma, \mu}^{{\bf x},{\boldsymbol  \alpha},{\bf s},{\boldsymbol  \beta}}  (g,1)>0$, if and only if:
 \begin{align*}
     &\text{ when } \mu>0 \text{ and } \mu_\partial>0, \eqref{seiberg1}+\eqref{seiberg2}+\eqref{seiberg3} \text{ hold };\\
     &\text{ when } \mu>0 \text{ and } \mu_\partial=0, \eqref{seiberg1}+\eqref{seiberg2} \text{ hold };\\
     &\text{ when } \mu=0 \text{ and } \mu_\partial>0, \eqref{seiberg1}+\eqref{seiberg3} \text{ hold }.
 \end{align*}
\end{proposition}
\begin{proof}
It's suffice to show $F=1$ and $\mu\mu_{\partial}>0$ case.
We fix $g_0$ is a uniform type I metric and we will also consider the case of $g=e^{\varphi}g_0$ for $\varphi\in C^\infty(\Sigma)$ to understand the behaviour of the correlation functions under conformal change. We shift $X_{g}$ to $X_{g_0}$, then $V^{\alpha_i}_{g,\eps}(z_i)$ becomes 
$\tilde{V^{\alpha_i}_{g_0,\eps}}(z_i):=\eps^{\alpha_i^2/2}  e^{\alpha_i  (c+\tilde{X_{g_0,\epsilon}}(z_i))}$, so
\[\tilde{ V^{\alpha_i}_{g_0,\eps}}(z_i)= e^{\alpha_ic+\frac{\alpha_i^2}{4}\varphi(z_i)+\frac{\alpha_i^2}{4}\Big(W_{g_0}(z_i)+W_{g_0}(\sigma(z_i))+4\pi G^{d\Sigma}_{g_0}(z_i,\sigma(z_i))\Big)} e^{\alpha_i \tilde{X_{g_0,\epsilon}}(z_i)-\frac{\alpha_i^2}{2}\E[\tilde{X_{g_0,\epsilon}}(z_i)^2]}(1+o(1))\]
\[\tilde{ V^{\frac{\beta_j}{2}}_{g_0,\eps}}(s_j)= e^{\frac{\beta_j}{2}c+\frac{\alpha_i^2}{8}\varphi(s_j)+\frac{\alpha_i^2}{4}W_{g_0}(s_j)} e^{\frac{\beta_j}{2} \tilde{X_{g_0,\epsilon}}(s_j)-\frac{\beta^2_j}{8}\E[\tilde{X_{g_0,\epsilon}}(s_j)^2]}(1+o(1))\]
as $\eps\to 0$, with the remainder being deterministic. 
Then first apply the Girsanov transform on $$\prod_{i=1}^ne^{\tilde{X_{g_0,\epsilon}}(z_i)-\frac{\alpha_i^2}{2}\E[\tilde{X_{g_0,\epsilon}}(z_i)^2]}\prod_{j=1}^me^{\frac{\beta_j}{2} \tilde{X_{g_0,\epsilon}}(s_j)-\frac{\beta^2_j}{8}\E[\tilde{X_{g_0,\epsilon}}(s_j)^2]}$$ we get
\[\begin{split} 
 & e^{C_\eps({\bf z},{\bf s})}
  \int_\R  e^{c(\sum_i\alpha_i+\sum\frac{\beta_j}{2}-Q\chi(\Sigma))}\E\Big[  \exp\Big( -\frac{Q}{4\pi}\int_{\Sigma} X_{g_0}R_{g}dv_g  -\frac{Q}{2\pi}\int_{\partial\Sigma} X_{g_0}k_{g}d\lambda_g - \mu  e^{\gamma c}
 \hat{Z}_\eps  - \mu_{\partial}  e^{\frac{\gamma}{2} c}
 \hat{Z}^{\partial}_\eps\Big) \Big]\,dc\, (1+o(1))
\end{split}\]
where 
\[\begin{gathered} 
\hat{Z}_\eps:= \eps^{\frac{\gamma^2}{2}}\int_{\Sigma} e^{\gamma (\tilde{X_{g_0,\eps}}+H_{g_0,\eps})}{\rm dv}_{g}\quad \hat{Z}^{\partial}_\eps:= \eps^{\frac{\gamma^2}{4}}\int_{\partial\Sigma} e^{\frac{\gamma}{2} (\tilde{X_{g_0,\eps}}+H_{g_0,\eps})}{\rm d\lambda}_{g}\\
H_{g_0,\eps}(x):=\sum_i 2\pi \alpha_i \tilde{G^{\Sigma}_{g_0,\eps}}(z_i,x)+\sum_j\pi\beta_j\tilde{G^{\Sigma}_{g_0,\eps}}(s_j,x), \\
\end{gathered}\]
And $C_\eps({\bf z},{\bf s})$ gather all other factor, it's easy to see $\lim_{\eps\to 0}C_\eps({\bf z},{\bf s})$ exists.
By applying Girsanov transform again just like in the proof of Proposition \ref{covconf}, we can get rid of the 
curvature dependence terms and this shifts the field $\tilde{X_{g_0,\eps}}$ in $\hat{Z}_\eps$ 
by $S(x)=-\tfrac{Q}{2}(\varphi(x)-m_{g_0}(\varphi))$: 
\[\begin{split} 
A_\eps= & e^{C_\eps({\bf z},{\bf s})+\frac{Q^2}{16\pi}||d\varphi||^2_{g_0}}  \int_{\R} e^{c(\sum_i\alpha_i+\sum_j\frac{\beta_j}{2}-Q\chi(\Sigma))}
\E\Big[  \exp\Big( - \mu  e^{\gamma c}
 \til{Z}_\eps  - \mu_{\partial}  e^{\frac{\gamma}{2}c}
 \til{Z}^{\partial}_\eps  \Big) \Big]\,dc\, (1+o(1))
\end{split}\]
where $\til{Z}_\eps:=\eps^{\frac{\gamma^2}{2}}\int_M e^{\gamma (\tilde{X_{g_0,\eps}}+H_{g_0,\eps}+S)}{\rm dv}_{g}$ and $\til{Z}^{\partial}_\eps:= \eps^{\frac{\gamma^2}{4}}\int_{\partial\Sigma} e^{\frac{\gamma}{2} (\tilde{X_{g_0,\eps}}+H_{g_0,\eps}+S)}{\rm d\lambda}_{g}$; 
By Lemma \ref{Greenesti},  $||H_{g_0,\eps}||_{L^\infty}<\infty$ t, we get that 
$\E[\til{Z}_\eps]<\infty$ and $\E[\til{Z}^{\partial}_\eps]<\infty$. Therefore we can find $A>0$ such that $\mathbb{P}(\til{Z}_\eps\leq A,\til{Z}^{\partial}_\eps\leq A )>0$. We therefore get 
\[ \Pi_{\gamma, \mu,\mu_\partial}^{{\bf z},{\boldsymbol  \alpha},{\bf s},{\boldsymbol  \beta}}  (g,1,\eps)\geq \beta_{\eps,{\bf z},{\bf s}} \int_{-\infty}^0 e^{c(\sum_i\alpha_i+\sum_j\frac{\beta_j}{2}-Q\chi(\Sigma)) - \mu  e^{\gamma c}A-\mu_\partial e^{\frac{\gamma}{2}c}A} \mathbb{P}(\til{Z}_\eps\leq A,\til{Z}^{\partial}_\eps\leq A )\,dc
\]
for some $\beta_{\eps,{\bf z},{\bf s}}>0$, and this is infinite if $\sum_i\alpha_i +\sum_j\frac{\beta_j}{2}-Q\chi(\Sigma)\leq 0$.\\
The next thing we need to check is that the following integrals are finite a.s. (otherwise, the partition function will be worth 0 ):
$$
\int_{\Sigma} e^{\gamma H_{g_0}(x)} M_{\gamma, g_0}(d x) \text { and } \int_{\partial \Sigma} e^{\frac{\gamma}{2} H_{g_0}(x)} M_{\gamma, g_0}^{\partial}(d x) .
$$
For a general metric $g$, we can use \eqref{Measurechange}. Proposition \eqref{finiteness} tells us that the integrals without $H$ are finite a.s. This means we can restrict ourselves to looking at what happens around the insertions, For a bulk insertion $\left(z_{i}, \alpha_{i}\right), e^{\gamma H_{g_0}(z)}$ behaves like $\frac{1}{ d_{g_0}(z,z_{i})^{\alpha_{i} \gamma}}$ around $z_{i}$, we then uniformize to the Poincare disk as in \eqref{Greenesti}. For $B\left(z_{i}, r\right)$ a small ball around $z_{i}$ of geodesic radius $r>0$, we have:
$$
\int_{B\left(z_{i}, r\right)} \frac{1}{ d_{g_0}(z,z_{i})^{\alpha_{i} \gamma}} M_{\gamma, g_0}(d x)<+\infty \text { a.s. } \Longleftrightarrow \alpha_{i}<Q .
$$
In the same way we can look at a boundary insertion $\left(s_{j}, \beta_{j}\right)$ and we have following \cite{HRV}  $$\int_{B\left(s_{j}, r\right) \cap \Sigma} \frac{1}{d_{g_0}(x,s_{j})^{\beta_{j} \gamma}} M_{\gamma, g_0}(d x)<+\infty \quad and \quad \int_{B\left(s_{j}, r\right) \cap  \partial\Sigma} \frac{1}{d_{g_0}(x,s_{j})^{\frac{\beta_{j} \gamma}{2}}} M_{\gamma, g_0}^{\partial}(d x)<+\infty\text { a.s. }  \Longleftrightarrow \beta_{j}<Q .$$ Therefore we have all the conditions of proposition \eqref{correlation}.
 \end{proof}
 
The proof of the previous proposition (adding a functional $F$ does not change anything) also shows the 
\begin{proposition}\label{covconf2}{\bf (Conformal anomaly and diffeomorphism invariance)}
Let 
$g_0$ and $g$ in the same conformal class, i.e. $ g=e^{\varphi}g_0$ for some $\varphi\in C^\infty(\Sigma)$. Then we have
\begin{equation}\label{confan} 
\log\frac{\Pi_{\gamma, \mu,\mu_\partial}^{{\bf z},{\boldsymbol  \alpha},{\bf s},{\boldsymbol  \beta}}  (g,F)}{\Pi_{\gamma, \mu,\mu_\partial}^{{\bf z},{\boldsymbol  \alpha},{\bf s},{\boldsymbol  \beta}}  (g_0,F(\cdot-\tfrac{Q}{2}\varphi))}= 
\frac{1+6Q^2}{96\pi}\Big(\int_{\Sigma}(||d\varphi||^2_{g_0}+2R_{g_0}\varphi) {\rm dv}_{g_0}+4 \int_{\partial \Sigma} k_{g_0}\varphi d\lambda_{g_0}\Big)-\Delta_{\alpha_i}\varphi(z_i)-\frac{1}{2}\Delta_{\beta_j}\varphi(s_j)
\end{equation}

Let $\psi:\Sigma\to \Sigma$ be an orientation preserving diffeomorphism. Then  
\[ \Pi_{\gamma, \mu,\mu_\partial}^{{\bf z},{\boldsymbol  \alpha},{\bf s},{\boldsymbol  \beta}}     (\psi^*g ,F)= \Pi_{\gamma, \mu,\mu_\partial}^{{\bf \psi(z)},{\boldsymbol  \alpha},{\bf \psi(s)},{\boldsymbol  \beta}}   (g,F(\cdot \circ \psi)) .\]
\end{proposition}

\subsection{the scaling relation in the boundary case}
In this section, we restrict ourselves to $(\Sigma^{\circ},\partial\Sigma)=(\H, \partial\H).$
For $(\mathbf{z},\mathbf{\alpha})$ and $(\mathbf{s},\mathbf{\beta})$ in the Seiberg bound, we define $G_{\epsilon}(\mathbf{z},\mathbf{s})=\Pi_{\gamma, \mu,\mu_\partial}^{{\bf z},{\boldsymbol  \alpha},{\bf s},{\boldsymbol  \beta}}  (g,1,\eps)$ and more generally for $\mathbf{x}=x_{1}, \cdots, x_{n} \in \H$ and $\mathbf{y}=y_{1}, \cdots, y_{m} \in \partial\H $
$$
G_{\epsilon}(\mathbf{x} ; \mathbf{z};\mathbf{s}):=\Pi_{\gamma, \mu,\mu_\partial}^{{\bf z},{\boldsymbol  \alpha},{\bf s},{\boldsymbol  \beta}}  (g,\prod _iV_{\gamma}(x_i),\eps)
$$ and
$$G^{\partial}_{\epsilon}(\mathbf{y} ; \mathbf{z};\mathbf{s}):=\Pi_{\gamma, \mu,\mu_\partial}^{{\bf z},{\boldsymbol  \alpha},{\bf s},{\boldsymbol  \beta}}  (g,\prod_iV_{\frac{\gamma}{2}}(y_i),\eps)$$
here we abuse the notation since $V_{\gamma}$ are not bounded, but it still makes sense, since inserting more vertex operators doesn't break the Seiberg bounds. The following scaling formula helps us get rid of metric dependence when we perform Gaussian integration by parts, we suppose $g$ is the uniform type I metric to remove the curvature dependence
\begin{proposition}[Scaling relation]\label{KPZboundary}
For all $\epsilon \geq 0$, 
$$
 \int_{\H}\mu \gamma G_{\epsilon}(x ; \mathbf{z};\mathbf{s}) dv_g(x)+\int_{\partial\H}\mu_{\partial} \frac{\gamma}{2}  G^{\partial}_{\epsilon}(y ; \mathbf{z};\mathbf{s}) d\lambda_g(y)=\left(\sum_{k=1}^{N} \alpha_{k}+\sum_{j=1}^{M} \frac{\beta_{j}}{2}-\chi(\H) Q\right) G_{\epsilon}(\mathbf{z};\mathbf{s})
$$
\end{proposition}
\begin{proof}
We only deal with the uniform type I metric $g$, we simply change c to $c^{\prime}=\gamma^{-1} \ln \mu+c$ that
$$
\begin{aligned}
&\Pi_{\gamma, \mu,\mu_\partial}^{{\bf x},{\boldsymbol  \alpha},{\bf s},{\boldsymbol  \beta}}  (g,1,\eps) =\int_{\mathbb{R}} e^{-\chi(\H) Q c} \mathbb{E}\left[\prod_{k=1}^{N} V_{\alpha_{k}, \epsilon}\left(z_{k}\right)\prod_{j=1}^{M} V_{\frac{\beta_j}{2}, \epsilon}\left(s_{j}\right) e^{-\mu \int_{\H} V_{\gamma, \epsilon}(x) dv_g( x)} e^{-\mu_{\partial} \int_{\partial\H} V_{\frac{\gamma}{2}, \epsilon}(y) d\lambda_g( y)}\right] d c \\
&=\mu^{-\frac{\sum_{k=1}^{N} \alpha_{k}+\sum_{j=1}^{M} \frac{\beta_{j}}{2}-\chi(\H) Q}{\gamma}} \int_{\mathbb{R}} e^{-\chi(\H) Q c'} \mathbb{E}\left[\prod_{k=1}^{N} V_{\alpha_{k}, \epsilon}\left(z_{k}\right)\prod_{j=1}^{M} V_{\frac{\beta_j}{2}, \epsilon}\left(s_{j}\right) e^{-\int_{\H} V_{\gamma, \epsilon}(x) dv_g(x)} e^{-\mu_{\partial} \mu^{-\frac{1}{2}}\int_{\partial\H} V_{\frac{\gamma}{2}, \epsilon}(y) d\lambda_g( y)}\right] d c'
\end{aligned}
$$
The identity follows by differentiating in $\mu$, more precisely, we remove $\mu^{-\frac{1}{2}}$ from $\mu^{-\frac{3}{2}}$ by change $c'$ to c backward.
\begin{align*}
   &\partial_\mu\Pi_{\gamma, \mu,\mu_\partial}^{{\bf z},{\boldsymbol  \alpha},{\bf s},{\boldsymbol  \beta}}  (g,1,\eps) \\&=-\frac{\sum_{k=1}^{N} \alpha_{k}+\sum_{j=1}^{M} \frac{\beta_{j}}{2}-\chi(\H) Q}{\gamma}\frac{1}{\mu}\Pi_{\gamma, \mu}^{{\bf z},{\boldsymbol  \alpha},{\bf s},{\boldsymbol  \beta}}  (g,1,\eps)+\mu^{-\frac{\sum_{k=1}^{N} \alpha_{k}+\sum_{j=1}^{M} \frac{\beta_{j}}{2}-\chi(\H) Q}{\gamma}} \int_{\mathbb{R}} e^{-\chi(\H) Q c'} \\&\times\mathbb{E}\left[\prod_{k=1}^{N} V_{\alpha_{k}, \epsilon}\left(z_{k}\right)\prod_{j=1}^{M} V_{\frac{\beta_j}{2}, \epsilon}\left(s_{j}\right) \frac{1}{2}\mu_{\partial}\mu^{-\frac{1}{2}}\frac{1}{\mu}\int_{\partial\H} V_{\frac{\gamma}{2}, \epsilon}(y) d\lambda_g( y)e^{-\int_{\H} V_{\gamma, \epsilon}(x) dv_g(x)} e^{-\mu_{\partial} \mu^{-\frac{1}{2}}\int_{\partial\H} V_{\frac{\gamma}{2}, \epsilon}(y) d\lambda_g( y)}\right] d c'\\
   &=-\frac{\sum_{k=1}^{N} \alpha_{k}+\sum_{j=1}^{M} \frac{\beta_{j}}{2}-\chi(\H) Q}{\gamma\mu}\Pi_{\gamma, \mu,\mu_\partial}^{{\bf z},{\boldsymbol  \alpha},{\bf s},{\boldsymbol  \beta}}  (g,1,\eps)+\int_{\partial\H}\frac{\mu_{\partial}}{2\mu}\Pi_{\gamma, \mu,\mu_\partial}^{{\bf z},{\boldsymbol  \alpha},{\bf s},{\boldsymbol  \beta}}  (g,V_{\frac{\gamma}{2}}(y),\eps)d\lambda_g(y)
\end{align*}
\end{proof}
By the Fusion estimate, we know $\sup_{\eps}G_{\epsilon}(x; \mathbf{z};\mathbf{s})\in L^1(\H)$ and $\sup_{\eps}G^{\partial}_{\epsilon}(y; \mathbf{z};\mathbf{s})\in L^1(\R)$ so we can send $\eps\to 0$ by the dominate convergence. For self-consistency, we gather the fusion rules for boundary Liouville in the next subsection.
\subsection{Fusion estimate}
We note $|x|_\epsilon:=|x|\vee \epsilon$.
\begin{proposition}\label{fusionestimate}
    Define  $V_{\alpha,\epsilon}(z)$, $z\in\H$ as bulk insertion and $V_{\frac{\beta}{2},\epsilon}(s)$, $s\in\R$ as boundary insertion, then we have the following bounds
    \begin{itemize}
        \item $V_{\alpha_1,\epsilon}(z_1)V_{\alpha_2,\epsilon}(z_2)\leq (|z_1-z_2|_ \epsilon)^{2(\Delta_{(\alpha_1+\alpha_2)\wedge Q}-\Delta_{\alpha_1}-\Delta_{\alpha_2})}$
        \item $V_{\alpha,\epsilon}(z)\leq (|z-\Bar{z}|_ \epsilon)^{\Delta_{(2\alpha)\wedge Q}-2\Delta_{\alpha}}$
        \item $V_{\frac{\beta_1}{2},\epsilon}(s_1)V_{\frac{\beta_2}{2},\epsilon}(s_2)\leq (|s_1-s_2|_ \epsilon)^{\Delta_{(\beta_1+\beta_2)\wedge Q}-\Delta_{\beta_1}-\Delta_{\beta_2}}$
    \end{itemize}
    Here we use a brief notation, the left-hand side should be understood as an appropriate correlation function, see \cite[Section 6]{KRV}.
\end{proposition}

\begin{remark}
This relation holds for any Riemann surfaces with boundaries since the Fusion estimate is only a local behavior and also holds if we replace the positive constant $\mu_{\partial}$ by a function $\mu_{\partial}:\partial \Sigma\to \mathbb{R}^+$ which is piecewise constant.
\end{remark}
Here we only estimate the singularity when a bulk insertion $(z,\gamma)$ approaches boundary insertion $(\beta,s)$ since other cases are much easier than this one. First, we recall some settings on LCFT on the upper half-plane, noting the dozz metric $g(x)=\frac{1}{|x|^4_+}$.
The upper half-plane GFF in the dozz metric is defined by \begin{align}
    \E[X(x)X(y)]=\ln(\frac{1}{|x-y||x-\bar{y}|})+2\ln(|x|_+)+2\ln(|y|_+)
\end{align}
define the bulk GMC measure as $M_{\gamma}(d^2z)=e^{\gamma X(z)-\frac{\gamma^2}{2}\E[X^2(z)]}\Big(\frac{|z|^2_{+}}{|z-\Bar{z}|}\Big)^{\frac{\gamma^2}{2}}\frac{1}{|z|^4_+}\dd^2z$ and boundary GMC measure as $M^{\partial}_{\gamma}(dz)=e^{\frac{\gamma}{2} X_{\mathbb{H}}(z)-\frac{\gamma^2}{8}\mathbb{E}[X_{\mathbb{H}}(z)^2]}\frac{1}{|z|^2_+}\dd z$.

Then we do the standard radial decomposition
$X(x)=x_{-2\ln(|X|)}+Y(x)$, where $t\to x_t$ is a Brownian motion $B_t$ from $0$ for $t\geq 0$ and we set $x_t=0$ for $t<0$.
We suppose $z\in \delta\H\D$ and $s\in (-\delta,\delta)$ with some small $\delta$, and all other insertions are bounded away from these two insertions.
Then what we need to estimate is the integrability of the following, we replace $\H$ by $\H\D$ and $\mu_\partial=0$ to get an upper bound
\begin{align}\label{L1estimate}
    z\to |z-\bar{z}|_{\epsilon}^{-\frac{\gamma^2}{2}}|z-s|_{\epsilon}^{-\gamma\beta}\E[(\int_{\H\D}\frac{1}{|x-z|_{\epsilon}^{\gamma^2}|x-\bar{z}|_{\epsilon}^{\gamma^2}|x-s|_{\epsilon}^{\gamma\beta}|x-\bar{x}|^{\frac{\gamma^2}{2}}}e^{\gamma X(x)-\frac{\gamma^2}{2}X(x)^2}d^2x)^{-q}]
\end{align}
We note the expectation part of RHS as a function $Q:A\to Q(A)$ for domain $A\subset\H\D$ with evaluation $A=\H\D$.
Here $q=\frac{\gamma+\frac{\beta}{2}+\sum^n_{i=1}\alpha_i+\sum_{j=1}^m\frac{\beta_j}{2}-Q}{\gamma}$ and we note $q'=\frac{\frac{\beta}{2}+\sum^n_{i=1}\alpha_i+\sum_{j=1}^m\frac{\beta_j}{2}-Q}{\gamma}$, we assume $q'$ is positive to meet the Seiberg bound. Since $\beta<Q$, we know that $q''=\frac{\sum^n_{i=1}\alpha_i+\sum_{j=1}^m\frac{\beta_j}{2}-\frac{Q}{2}}{\gamma}>0$

we have the following decomposition of $M_\gamma(d^2z)$
\begin{align}
    \int_{\delta\H\D}f(z)M_{\gamma}(d^2z)=\int_{0}^{\pi}\int_{-2\ln\delta}^{\infty}f(e^{-\frac{t}{2}+i\sigma})e^{\gamma (x_t-\frac{Q}{2}t)}\mu_{Y}(dt,d\sigma)
\end{align}
where $\mu_Y(dt,d\sigma)=\frac{e^{\gamma Y(e^{-\frac{t}{2}+i\sigma})-\frac{\gamma^2}{2}\E[Y(e^{-\frac{t}{2}+i\sigma})^2]}}{|\sin(\sigma)|^{\frac{\gamma^2}{2}}}dt d\sigma$. 

\begin{itemize}
    \item Case 1. When $2\gamma>Q$ and $\beta>0$,we follow the notation and the proof of part 1 in \cite[proposition 5.3]{KRV}. We define the annulus $\mathcal{A}=\{x\mid|x-\bar{z}|\in (|z-\bar{z}|,\delta)\}\cap\H$, In this region, we have $|x-z|\leq|x-\bar{z}|$ and $|x-s|\leq 2|x-\bar{z}|1_{\{|x-\bar{z}|>|s-\bar{z}|\}}+2|s-\bar{z}|1_{\{|x-\bar{z}|\leq|s-\bar{z}|\}}$. Now we define
    \begin{align}
        y_t=x_t+\frac{(\gamma+\gamma-Q)}{2}t-\frac{\beta}{2} (t\wedge -2\ln|z-s|_\epsilon)
    \end{align}
    and 
\begin{align}
& M_n=\left\{\max _{s \in\left[0,-2\ln \left|z-\bar{z}\right|_\epsilon\right]} y_s \in[n-1, n]\right\}, \quad n \geqslant 1, \\
& M_0=\left\{\max _{s \in\left[0,-2\ln \left|z-\bar{z}\right|_\epsilon\right]} y_s \leqslant 0\right\} .
\end{align}
    Then 
    \begin{align}
        Q(\H\D)\leq Q(\mathcal{A})\leq C\sum_n\E[1_{M_n}(\int_{-2\ln\delta}^{-2\ln|z-\bar{z}|_\epsilon}\int_{\pi/6}^{5/6\pi} e^{\gamma y_t} \mu_Y(dt,d\sigma))^{-q}]:=\sum_n A_n
    \end{align}
    Then there are two subcases:
    
    Subcase 1.  When $-2\ln|z-\bar{z}|_\epsilon>-2\ln|z-s|_\epsilon$, then we use the item 3 \cite[Lemma 6.5]{KRV} to bound \eqref{fusionestimate} by $|z-\bar{z}|_\epsilon^{\frac{Q^2}{4}-2}|z-s|_\epsilon^{-\Delta_\beta}$. The we use polar coordinate, we see $({|\sin(\sigma)|})^{\frac{Q^2}{4}-2}d\sigma$ is integrable and $r^{1+\frac{Q^2}{4}-2-\Delta_\beta}dr$ is also integrable since $\Delta_\beta<Q^2/4$.

    Subcase 2. When $-2\ln|z-\bar{z}|_\epsilon\leq-2\ln|z-s|_\epsilon$, we rewrite $y_t$ as $x_t+\frac{(\gamma+\gamma+\beta-Q)}{2}t$. Now we use the item 1 in \cite[Lemma 6.5]{KRV}, we bound \eqref{fusionestimate} by $|z-\bar{z}|_\epsilon^{-\frac{\gamma^2}{2}}|z-s|_\epsilon^{-\gamma\beta}|z-\bar{z}|_\epsilon^{\frac{(2\gamma+\beta-Q)^2}{4}}$. As above, we examine the integrability of $r $ and $\sigma$ separately. For $\sigma$, we know $|\sin(\sigma)|^{\frac{(2\gamma+\beta-Q)^2}{4}-\frac{\gamma^2}{2}}d\sigma$ is integrable since $\frac{(2\gamma+\beta-Q)^2}{4}-\frac{\gamma^2}{2}=\frac{Q^2}{4}-2+\frac{(2\gamma-Q)\beta}{2}+\frac{\beta^2}{4}>\frac{Q^2}{4}-2$. For $r$, we have $1+\frac{(2\gamma+\beta-Q)^2}{4}-\frac{\gamma^2}{2}-\beta\gamma=\frac{Q^2}{4}-1-\Delta_\beta>-1.$
    
    \item Case 2. When $2\gamma>Q$ and $\beta\leq 0$, then this case is reduced to the bulk point fusion to the real line, the singularity is $|z-\Bar{z}|_\epsilon^{\frac{Q^2}{4}-2}$ and $Q^2/4-2\in (-2/3,1)$, which is integrable.

    \item Case 3. When $2\gamma\leq Q$ and $2\gamma+\beta>Q$, we define $\mathcal{A}=\{x\mid |x-s|\in (|z-s|,\delta)\} \cap \H$, then we get $y_t=x_t+\frac{(\gamma+\gamma+\beta-Q)}{2}t$. Repeat the proof \cite[proposition 5.1]{KRV}, we know The singularity behaves like $|x-s|_\epsilon^{\frac{(2\gamma+\beta-Q)^2}{4}-\gamma\beta}$, use polar coordinate we can bound this by $r^{1+\frac{(2\gamma+\beta-Q)^2}{4}-\gamma(Q-2\gamma)}$, which is integrable.
    \item Case 4. when $2\gamma\leq Q$ and $2\gamma+\beta<Q$, we can simply bound it by $|z-\bar{z}|_{\epsilon}^{-\frac{\gamma^2}{2}}|z-s|_\epsilon^{-\beta\gamma}$. When using polar coordinate, it reads $|\sin(\sigma)|^{-\frac{\gamma^2}{2}}r^{1-\frac{\gamma^2}{2}-\beta\gamma}drd\sigma$. Then we see $1-\frac{\gamma^2}{2}-\beta\gamma>1-\frac{\gamma^2}{2}-(Q-2\gamma)\gamma=\gamma^2-1>-1$.
\end{itemize}
\begin{proposition}
    For ${\bf x}=(x_1,..,x_n)$ and ${\bf y}=(y_1,...,y_m)$, then the function $({\bf x},{\bf s})\to G(\bf x, \bf y,\bf z,\bf s)$ is in $L^1(\H^n\times \R^m)$.
\end{proposition}
\begin{proof}
We define the quantum area $A=M_{\gamma,g}(\H)$ and quantum length $L=M_{\frac{\gamma}{2}}(\R)$, then we hope to show 
\begin{align}\label{multiple}
    \langle A^n L^m \prod_{k=1}^NV_{\alpha_i}(z_i)\prod_{j=1}^MV_{\frac{\beta}{2}}(s_j) e^{-\mu A}e^{-\mu_\partial L}\rangle
\end{align}
We notice that $$\sum_{m,n=0}^{\infty}\frac{(\frac{\mu}{2}A)^n (\frac{\mu_\partial}{2}L)^m}{n!m!}\leq e^{\frac{\mu}{2}A}e^{\frac{\mu_\partial}{2}L}$$
\end{proof}
then we see $\eqref{multiple}\leq\langle \prod_{k=1}^NV_{\alpha_i}(z_i)\prod_{j=1}^MV_{\frac{\beta}{2}}(s_j) e^{-\mu/2 A}e^{-\mu_\partial/2 L}\rangle$, which is finite due to the Seiberg bound.
\section{Markov property of \texorpdfstring{$X^{\Sigma}_{g}$}{}.}
This section also needs a version of the Gaussian free field with mixed boundary conditions. Decompose the boundary $\partial\Sigma$ as a disjoint union $\partial_1\Sigma\cup\partial_2\Sigma$, the mixed condition problem is
\begin{equation}\label{mix}
    \Delta f=\lambda f \text { in } \Sigma ;  \quad \partial_n f=0 \text{ on } \partial_1\Sigma, \quad \partial_n f=0 \text{ on } \partial_2\Sigma
\end{equation}
We can check this operator with its domain is symmetric and positive, and admit a self-adjoint extension, note as $\Delta^{mix}$. Then the theorem \eqref{spectrum} is still true, so we can define the Gaussian free field by its covariance given by the Green function in \eqref{mix}.

First, we show a decomposition in the Law of Neumann GFF $X^{\Sigma}_{g_0}$ under the uniform type I metric $g_0$.
\begin{proposition}
Suppose $\Sigma$ is a compact Riemann surface with boundaries. For the Neumann GFF $X^{\Sigma}_{g_0}$ on $\Sigma$, we have the following decomposition in law,
\[X^{\Sigma}_{g_0}\stackrel{law}{=}X^{\Sigma_1}_{g_0}+X^{\Sigma_2}_{g_0}+P\varphi^{\Sigma}_{g_0}-c^{\Sigma}_{g_0}\]
where $\varphi^{\Sigma}_{g_0}$ is the restriction of $X^{\Sigma}_{g_0}$ on the circles $\mathcal{C}$ of $\Sigma$, which separates $\Sigma$ into two parts $\Sigma_1$ and $\Sigma_2$, and $X^{\Sigma_i}_{g_0}$ is the mixed boundary condition GFF on $\Sigma_i$, i.e. $X^{\Sigma_i}_{g_0}=0$ on $\mathcal{C}$ and $\partial_{n_{g_0}}X^{\Sigma_i}_{g_0}=0$ on those boundary $\partial \Sigma_i$ of $\Sigma_i$ which are original boundary of $\Sigma$. $P\varphi^{\Sigma}_{g_0}$ is the harmonic extension with $\partial_{n_{g_0}}P\varphi^{\Sigma}_{g_0}=0$ on $\partial {\Sigma}_i$ and $c^{\Sigma}_{g_0}=m^{\Sigma}_{g_0}(X^{\Sigma_1}_{g_0}+X^{\Sigma_2}_{g_0}+P\varphi^{\Sigma}_{g_0})$, here $X^{\Sigma_1}_{g_0}$, $X^{\Sigma_2}_{g_0}$, and $P\varphi^{\Sigma}_{g_0}$ are all independent.
\end{proposition}
\begin{figure}[H] 
\centering 
\includegraphics[width=0.7\textwidth]{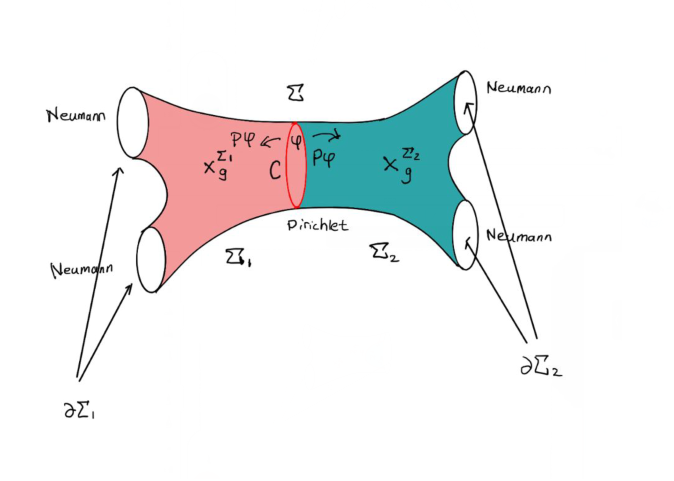} 
\caption{Markov property of circle cutting.} 
\label{Fig.main1} 
\end{figure}

\begin{proof}
First, we use the double trick to get $d\Sigma$, then we cut $d\Sigma$ along both $\mathcal{C}$ and $\sigma(\mathcal{C})$ into two pieces $d\Sigma_1$ and $d\Sigma_2$. By Poisson kernel method \cite[appendix D]{GKRV21}, we can show that 
\[X^{d\Sigma}_{g_0}\stackrel{law}{=}X^{d\Sigma_1}_{g_0}+X^{d\Sigma_2}_{g_0}+P\varphi^{d\Sigma}_{g_0}-c^{d\Sigma}_{g_0}\]
where $X^{d\Sigma_i}_{g_0}$ is the Dirichlet GFF on $d\Sigma_i$, $P\varphi^{d\Sigma}_{g_0}$ is the harmonic extension of restriction $X^{d\Sigma}_{g_0}$ on $\mathcal{C}\cup\sigma(\mathcal{C})$, and $c^{d\Sigma}_{g_0}=m_{g_0}^{d\Sigma}(X^{d\Sigma_1}_{g_0}+X^{d\Sigma_2}_{g_0}+P\varphi^{d\Sigma}_{g_0})$, and $X^{d\Sigma_1}_{g_0}$, $X^{d\Sigma_2}_{g_0}$, and $P\varphi^{d\Sigma}_{g_0}$ are all independent. Now we define $X^{\Sigma_i}_{g_0}(x): =\frac{X^{d\Sigma_i}_{g_0}(x)+X^{d\Sigma_i}_{g_0}(\sigma(x))}{\sqrt{2}}$ and $P\varphi^{\Sigma}_{g_0}(x):=\frac{P\varphi^{d\Sigma}_{g_0}(x)+P\varphi^{d\Sigma}_{g_0}(\sigma(x))}{\sqrt{2}}$. By noticing that $c^{\Sigma}_{g_0}=\sqrt{2}c^{d\Sigma}_{g_0}$, we can show the required property easily.
\end{proof}
Now we extend the Markov property to a general metric $g$ which is in the conformal class $g_0$.
\begin{proposition}
For the Neumann GFF $X^{\Sigma}_{g}$, we have the following decomposition in law,
\begin{equation}\label{Marg}
X^{\Sigma}_{g}\stackrel{law}{=}X^{\Sigma_1}_{g}+X^{\Sigma_2}_{g}+P\varphi^{\Sigma}_{g}-c^{\Sigma}_{g}  
\end{equation}
where $\varphi^{\Sigma}_{g}$ is the restriction of $X^{\Sigma}_{g}$ on the circles $\mathcal{C}$ of $\Sigma$, which separate $\Sigma$ into two parts $\Sigma_1$ and $\Sigma_2$, and $X^{\Sigma_i}_{g}$ is the mixed boundary GFF on $\Sigma_i$, i.e. $X^{\Sigma_i}_{g}=0$ on $\mathcal{C}$ and $\partial_{n_{g}}X^{\Sigma_i}_{g}=0$ on those boundary $\partial \Sigma_i$ of $\Sigma_i$ which are original boundary of $\Sigma$. $P\varphi^{\Sigma}_{g}$ is the harmonic extension with $\partial_{n_{g}}P\varphi^{\Sigma}_{g}=0$ on $\partial \Sigma_i$ and $c^{\Sigma}_{g}=m^{\Sigma}_{g}(X^{\Sigma_1}_{g}+X^{\Sigma_2}_{g}+P\varphi^{\Sigma}_{g})$, here $X^{\Sigma_1}_{g}$, $X^{\Sigma_2}_{g}$, and $P\varphi^{\Sigma}_{g}$ are all independent.
\end{proposition}
\begin{proof}
First, we know 
\begin{equation}\label{Marg0}
X^{\Sigma}_{g}\stackrel{law}{=}X^{\Sigma_1}_{g_0}+X^{\Sigma_2}_{g_0}+P\varphi^{\Sigma}_{g_0}-m^{\Sigma}_{g}(X^{\Sigma_1}_{g_0}+X^{\Sigma_2}_{g_0}+P\varphi^{\Sigma}_{g_0})   
\end{equation} 
The main difficulty of the Markov property comes from that we renormalize $X^{\Sigma}_{g}$ has $0$ average on $\Sigma$. If we only consider the up to constant level, we have $\overline{X^{\Sigma}_{g}}\stackrel{law}{=}X^{{\Sigma}_1}_{g}+X^{\Sigma_2}_{g}+\overline{P\varphi^{\Sigma}_{g}}$, for example, see \cite[section 4.1.5]{DMS}. So we only need to check the right side of \eqref{Marg} and right hand of \eqref{Marg0} has the same distribution. The key ingredient is that this kind of mixed condition GFF are same in the same conformal class. And what left to check is 
\[\mathbb{E}\Big[\Big(P\varphi^{\Sigma}_{g_0}-m_g(P\varphi^{\Sigma}_{g_0})\Big)(x)\Big(P\varphi^{\Sigma}_{g_0}-m_g(P\varphi^{\Sigma}_{g_0})\Big)(y)\Big]=\mathbb{E}\Big[\Big(P\varphi^{\Sigma}_{g}-m_g(P\varphi^{\Sigma}_{g})\Big)(x)\Big(P\varphi^{\Sigma}_{g}-m_g(P\varphi^{\Sigma}_{g})\Big)(y)\Big]\]
Which is obviously true since we have $P\varphi^{\Sigma}_{g_0}-m_g(P\varphi^{\Sigma}_{g_0})\stackrel{law}{=}P\varphi^{\Sigma}_{g}-m_g(P\varphi^{\Sigma}_{g})$.
\end{proof}
We may think about another kind of cutting, instead of cutting a hole circle in the interior of the surface, we cut a half-circle near the boundary, i.e. two ends are both on the boundaries. Then perform again the strategy above, we get.
\begin{corollary}
The same thing still holds if we cut a half-circle which two ends are on the boundaries.
\end{corollary}
\begin{figure}[H] 
\centering 
\includegraphics[width=0.7\textwidth]{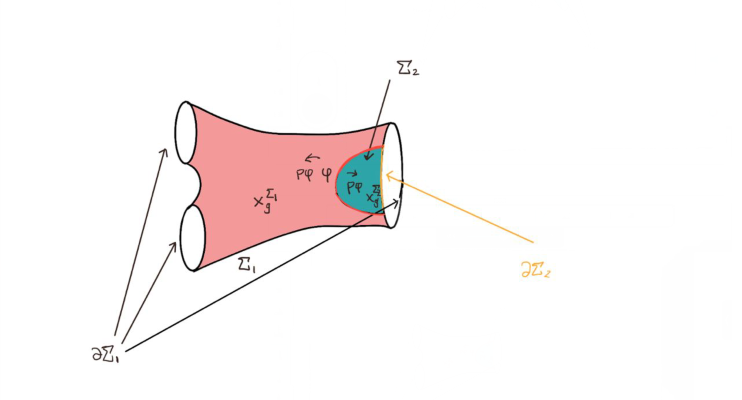} 
\caption{Markov property of half circle cutting.} 
\label{Fig.main2} 
\end{figure}

\section{Link with previous works on disk and annulus}\label{Diskannulus}
In the previous work, the disk topology in \cite{HRV} and the annulus topology in \cite{Remy}, they actually choose the type II uniform metric, i.e. the Gaussian curvature is 0 (flat) and the geodesic curvature is constant. But we can still unify these topology in our work, although their $\chi(\Sigma)=2(1-{\bf g})-k \geq 0$. 

In the case of a disk, we start by choosing the round metric $g=\frac{4}{(1+|z|^2)^2}|dz|^2$ on the sphere, which has Gaussian curvature 2. Then we cut the sphere into two hemispheres along the equator. The equator is a geodesic curve, so we get a type I uniform metric on the hemisphere, which is topological equivalent to a disk.

In the case of an annulus, we start from a torus with a flat metric $g_{\mathbb{A}}=\frac{|dz|^2}{|z|^2}$. Then we cut along two geodesic circles which separate the torus into two flat cylinders. So the cylinder has both 0 Gaussian curvature and 0 geodesic curvature, which means we get a type I uniform metric on the cylinder.

So after inserting enough vertex operators to make sure the correlation function is convergent, we can perform again what we did in the previous sections to construct the Liouville conformal field theory on disk and annulus by choosing a uniform type I metric.


\hspace{10 cm}


\begin{thebibliography}{20}
\bibitem[Po]{Pol}
A.M. Polyakov A.M., \emph{Quantum geometry of bosonic strings}, Phys. Lett. 
\textbf{103B} 207 (1981).
 
\bibitem[DRV]{DRV}
F. David, R. Rhodes, V. Vargas, \emph{Liouville Quantum Gravity on complex tori},
\emph{Journal of Mathematical physics} {\bf 57}, 022302 (2016). 


\bibitem[OPS]{OPS}  B. Osgood, R. Phillips, P. Sarnak, \emph{Extremals of determinants of Laplacians.}
J. Funct. Anal. \textbf{80} (1988), no. 1, 148--211.
\bibitem[Sar]{Sar}
P.Sarnak, \emph{Determinants of Laplacians.} Communications in Mathematical Physics Volume 110, pages113–120 (1987)
\bibitem[GRV]{GRV}  Guillarmou, C., Rhodes, R. Vargas, V. \emph{Polyakov’s formulation of 2d bosonic string theory.} Publ. math.IHES 130, 111–185 (2019). 
\bibitem[GKRV]{GKRV21} Colin Guillarmou, Antti Kupiainen, Rémi Rhodes, Vincent Vargas. \emph{Segal's axioms and bootstrap for Liouville Theory}.
\bibitem[Be]{Ber} 
N. Berestycki, \emph{An elementary approach of Gaussian multiplicative chaos},  Electron. Commun. Probab.  vol 22 (2017), no. 27, 12 pp.
\bibitem[RS1]{RS1}
D. Ray and I. M. Singer, R-torsion and the Laplacian on Riemannian manifolds, Adv.
Math. 7 (1971), 145–210.
\bibitem[RS2]{RS2} 
D. Ray and I. M. Singer, Analytic torsion for complex manifolds, Ann. Math. 98 (1973),
154–177
\bibitem[Mir07]{Mir07}
Mirzakhani, M. \emph{ Simple geodesics and {W}eil-{P}etersson volumes of moduli spaces of bordered {R}iemann surfaces.} {\em Invent. Math. 167}, 1 (2007), 179--222.
\bibitem[DMS]{DMS} Duplantier B., Miller J., Sheffield S.\emph{Liouville quantum gravity as a mating of trees}, arXiv:1409.7055.
\bibitem[DKRV]{DKRV} 
David F., Kupiainen A., Rhodes R., Vargas V.\emph{Liouville quantum gravity on the Riemann sphere.}, \emph{Commun. Math. Phys.}, \textbf{342}: 869-907 (2016).
\bibitem[KRV]{KRV}
Kupiainen, A., Rhodes, R. Vargas, V. Local Conformal Structure of Liouville Quantum Gravity. Commun. Math. Phys. 371, 1005–1069 (2019). https://doi.org/10.1007/s00220-018-3260-3
\bibitem[Remy]{Remy}
Guillaume Remy. \emph{Liouville quantum gravity on the annulus.} \emph{Journal of Mathematical Physics.} 59, 082303 (2018).
\bibitem[HRV]{HRV}
Huang Y., Rhodes R., Vargas V. \emph{Liouville quantum gravity on the unit disk.} \emph{Ann. Inst. H. Poincaré Probab. Statist. } 54(3): 1694-1730 (August 2018).
\bibitem[Alva]{Alva}
Alvarez, Orlando \emph{Theory of Strings with Boundaries: Fluctuations, Topology, and Quantum Geometry.}
\emph{Nucl. Phys. B}, 216, page 125--184, 1983
\bibitem[Lun]{Lun}
R. E.Lundelius,
\emph{Asymptotics of the determinant of the Laplacian on hyperbolic surfaces of finite volume.} \emph{Duke Math. J.}  71(1): 211-242 (July 1993).
\bibitem[Kim]{Kim}
Young-Heon Kim. \emph{Surfaces with boundary: Their uniformizations, determinants of Laplacians, and isospectrality.} \emph{Duke Math. J.} 144(1): 73-107 (15 July 2008). 
\bibitem[BCPCP]{BCPCP}
Steven K.Blau, Minot Clements, Stephen Della Pietra, Steven Carlip,Vincent Della Pietra. \emph{The string amplitude on surfaces with boundaries and crosscaps}. \emph{Nuclear Physics B
Volume 301}, Issue 2, 2 May 1988, Pages 285-303
\bibitem[DR]{DR}
J. Dodziuk and Burton Randol
\emph{Lower bounds for $\lambda_1$ on a finite-volume hyperbolic manifold}, Journal of Differential Geometry, 1986, 24. pages 133-139
\bibitem[DPRS]{DPRS}
Dodziuk, J., Pignataro, T., Randol, B.,Sullivan, D. (1987). Estimating small eigenvalues of Riemann surfaces.
\bibitem[Wo]{Wo} S. Wolpert, \emph{Asymptotics of the spectrum and the Selberg zeta function on the space of Riemann surfaces.} Comm. Math. Phys. \textbf{112} (1987), no. 2, 283-315.
\bibitem[RoVa]{RoVa} 
R. Robert, V. Vargas, \emph{Gaussian multiplicative chaos revisited}, Ann. Probab. {\bf 38} 
605-631 (2010). 
\bibitem[DRSV1]{Rnew7}
B. Duplantier, R. Rhodes, S. Sheffield, V. Vargas, \emph{Critical Gaussian multiplicative chaos: convergence of the derivative martingale}, Annals of Probability \textbf{42} (2014), no5, 1769--1808. 


\bibitem[DRSV2]{Rnew12}
Duplantier  B., Rhodes R., Sheffield S., Vargas V.: Renormalization of Critical Gaussian Multiplicative Chaos and KPZ formula, \emph{Commun. Math. Phys.} \textbf{330} (2014), no 1,  283--330.
\bibitem[Cha]{Cha}
I. Chavel, Eigenvalues in Riemannian Geometry. Volume 115, Academic Press, 2nd Edition (1987)
\bibitem[AHS20]{AHS20}
Morris Ang, Nina Holden, Xin Sun, Conformal welding of quantum disks, 	arXiv:2009.08389 [math.PR]
\bibitem[AHS21]{AHS21}
Morris Ang, Nina Holden, Xin Sun,
Integrability of SLE via conformal welding of random surfaces, 	arXiv:2104.09477 [math.PR]
\bibitem[ARS21]{ARS21}
Morris Ang, Guillaume Remy, Xin Sun,
FZZ formula of boundary Liouville CFT via conformal welding, arXiv:2104.09478 [math.PR]
\bibitem[AS]{AS}
Morris Ang, Xin Sun, Integrability of the conformal loop ensemble, 	arXiv:2107.01788 [math-ph].
\bibitem[Bo]{Bo}
David Borthwick, Spectral Theory
Basic Concepts and Applications, https://doi.org/10.1007/978-3-030-38002-1
\bibitem[Be]{Be}
Pierre H. B\'erard, Spectral Geometry: Direct and Inverse Problems, Lecture Notes in Mathematics, https://doi.org/10.1007/BFb0076330.
\bibitem[Chang]{Chang}
Sun-Yung Alice Chang, Non-linear Elliptic Equations in Conformal Geometry. ISBN print 978-3-03719-006-7, ISBN online 978-3-03719-506-2
DOI 10.4171/006
\bibitem[Wu]{Wu}
Baojun Wu, conformal bootstrap on annulus in Liouville CFT. arxiv
\bibitem[BW]{BW}
Guillaume Baverez, Mo Dick Wong, Fusion asymptotics for Liouville correlation functions, https://arxiv.org/abs/1807.10207
\end{thebibliography}
\end{document}